\documentclass[10pt, a4paper]{article} % use larger type; default would be 10pt

\usepackage[utf8]{inputenc} % set input encoding (not needed with XeLaTeX)

\usepackage[a4paper,margin=2.54cm]{geometry}
\usepackage{graphicx} % support the \includegraphics command and options
\usepackage{subcaption}
\usepackage{tikz}
\usepackage{pgf}
\usepackage{pgfplots,  pgfplotstable}
\pgfplotsset{compat=1.5}
\usepgfplotslibrary{external}
\usepgfplotslibrary{statistics}
\usetikzlibrary{patterns}
\usepackage{booktabs} % for much better looking tables
\usepackage{array} % for better arrays (eg matrices) in maths
\usepackage{paralist} % very flexible & customisable lists (eg. enumerate/itemize, etc.)
\usepackage{verbatim} % adds environment for commenting out blocks of text & for better verbatim
\usepackage{amsmath}

\usepackage{amsfonts}
\usepackage{amssymb}
\usepackage{amsthm}
\usepackage{algorithm}
\usepackage{algorithmic}
\usepackage{comment}
\usepackage{todonotes}
\usepackage{multirow}
\usepackage{tabularx}

\usepackage{algorithmic}
\usepackage{algorithm}
\usepackage{multicol}

\usepackage{url}

\newtheorem{theorem}{Theorem}
\newtheorem{lemma}{Lemma}

\usepackage{fancyhdr} % This should be set AFTER setting up the page geometry
\usepackage{sectsty}
\allsectionsfont{\sffamily\mdseries\upshape} % (See the fntguide.pdf for font help)
\usepackage[nottoc,notlof,notlot]{tocbibind} % Put the bibliography in the ToC
\title{Symmetric separable convex resource allocation problems with structured disjoint interval bound constraints}
\author{Martijn H. H. Schoot Uiterkamp \\ Tilburg University }
\begin{document}
\maketitle

\begin{abstract}
Motivated by the problem of scheduling electric vehicle (EV) charging with a minimum charging threshold in smart distribution grids, we introduce the resource allocation problem (RAP) with a symmetric separable convex objective function and disjoint interval bound constraints. In this RAP, the aim is to allocate an amount of resource over a set of $n$ activities, where each individual allocation is restricted to a disjoint collection of $m$ intervals. This is a generalization of classical RAPs studied in the literature where in contrast each allocation is only restricted by simple lower and upper bounds, i.e., $m=1$. We propose an exact algorithm that, for four special cases of the problem, returns an optimal solution in $O \left(\binom{n+m-2}{m-2} (n \log n + nF) \right)$ time, where the term $nF$ represents the number of flops required for one evaluation of the separable objective function. In particular, the algorithm runs in polynomial time when the number of intervals $m$ is fixed. Moreover, we show how this algorithm can be adapted to also output an optimal solution to the problem with integer variables without increasing its time complexity. Computational experiments demonstrate the practical efficiency of the algorithm for small values of $m$ and in particular for solving EV charging problems.

%Resource allocation problems (RAPs) are known to have a nice majorization property, namely that there exists a solution to the problem that is optimal for all Schur-convex objective functions simultaneously. In this paper, we consider an extension of the RAP with indicator variables and investigate whether the majorization property also holds for this new problem. We find this not to be the case, but do show that the property of the original RAP can be exploited to obtain an $O (n \log n + Cn)$ algorithm for the extension, where $C$ is the cost of evaluating the given Schur-convex objective function once.
\end{abstract}

\section{Introduction}

\subsection{Resource allocation problems with disjoint constraints and EV charging}

The resource allocation problem (RAP) is a classical problem in the operations research literature with many applications. In its most basic version, also referred to as the simple RAP, the problem asks for an allocation of a given amount of resource $R$ over a set $N := \lbrace 1, \ldots, n \rbrace$ of activities, subject to lower and upper bounds $l_i$ and $u_i$ on each allocation $x_i$ to an activity $i \in N$. The goal is to select an allocation that minimizes the sum of individual costs of this allocation. In this paper, we consider cost functions of the form $\sum_{i \in N} \phi(x_i + b_i)$, where $\phi \colon \mathbb{R} \rightarrow \mathbb{R}$ is a continuous convex function and $b \in \mathbb{R}^n$ acts as a shift vector. This means that the simple RAP can be formulated as follows:
\begin{align*}
\text{Simple RAP} \colon \ \min_{x \in \mathbb{R}^n} \ & \sum_{i \in N} \phi(x_i + b_i) \ \\
\text{s.t. } & \sum_{i \in N} x_i = R; \\
& l_i \leq x_i \leq u_i, \quad i \in N. 
\end{align*}
This problem has applications in many different fields such as finance, telecommunications, and machine learning (see \cite{Patriksson2008} for a survey). Many efficient methods exist to solve simple RAPs and we refer to \cite{Patriksson2015,SchootUiterkamp2022} for recent overviews of such methods and further problem properties.

In most studied extensions and variations of RAPs, each individual allocation is restricted to a single closed interval. In this paper, we study a generalization of the simple RAP where instead the feasible region for each variable is the union of $m$ closed intervals with $m > 1$. We refer to this problem as the RAP with disjoint interval bound constraints (RAP-DIBC):
\begin{align}
\text{RAP-DIBC} \colon \ \min_{x \in \mathbb{R}^n} \ & \sum_{i \in N} \phi(x_i + b_i) \nonumber \\
\text{s.t. } & \sum_{i \in N} x_i = R; \label{eq_res} \\
& x_i \in \cup_{j \in M} [l_{i,j} , u_{i,j}], \quad i \in N, \label{eq_box}
\end{align}
where, $M := \lbrace 1, \ldots, m \rbrace$, and $l, u \in \mathbb{R}^{n \times m}$. We assume without loss of generality that for each $i \in N$ the intervals $[l_{i,j}, u_{i,j}]$ are disjoint and that the vector $b$ is non-increasing.

Our motivation for studying this problem stems from its application in decentralized energy management (DEM) \cite{Siano2014, Esther2016}. In DEM, the goal is to optimize the simultaneous energy consumption of multiple devices within, e.g., a neighborhood. Within a DEM system, devices optimize their own consumption locally and the control system coordinates the local optimization of these devices to optimize certain neighborhood objectives (as opposed to other paradigms such as centralized energy management).

In particular, we are interested in the local optimization of a specific device class within DEM, namely the scheduling of electric vehicles (EVs) that adher to a minimum-charging threshold. Such a threshold means that, at any given moment, the EV is either idle (charges at rate zero) or charges within a particular range of rates. This is primarily due to technical constraints on EV batteries that prevent them from charging at rates very close to zero \cite{Young2013}. Moreover, charging at low rates is generally more inefficient and thus should be avoided \cite{ApostolakiIosifidou2017}. 

Mathematically, the problem can be stated as follows. We consider a division of the scheduling horizon into $T$ time intervals of length $\Delta t$. For each $t \in \lbrace 1, \ldots, T \rbrace$, we introduce the variable $x_t$ that denotes the power consumption of the EV during time interval $t$. Moreover, the parameter $p_t$ denotes the remaining static power consumption during interval $t$. We assume the total energy demand of the EV to be known on forehand and denote this value by $R$. The minimum-threshold restriction entails a minimum threshold $X^{\min}$ and a maximum charging rate $X^{\max}$ so that during an interval $t$ either the EV is idle ($x_t = 0$) or charges at a rate in between the threshold and the maximum rate ($X^{\min} \leq x_t \leq X^{\max}$). The objective is to minimize the peak consumption of the combined EV and static load, which can be expressed by the function $\sum_{t=1}^T \phi(x_t + p_t)$. Common choices for $\phi$ are the quadratic function, absolute value function, or hinge max functions (see also Section~5.2 of \cite{SchootUiterkamp2022}). Summarizing, this leads to the following formulation of the EV scheduling problem:
\begin{align*}
\text{Min-Thres-EV}\colon \ \min_{x \in \mathbb{R}^T} \ & \sum_{t=1}^T \phi(x_t + p_t) \\
\text{s.t. } & \sum_{t=1}^T \Delta t x_t = R; \\
& x_t \in \lbrace 0 \rbrace \cup [X^{\min}, X^{\max}], \quad t \in \lbrace 1,\ldots, T \rbrace.
\end{align*}
Note, that this problem is an instance of RAP-DIBC with $m=2$ and $l_{i,1} = u_{i,1} = 0$, $l_{i,2} = X^{\min}$, and $u_{i,2} = X^{\max}$ for all $i \in N$.

%In this application, we consider the problem of scheduling the charging of an electric vehicle within a low-voltage distribution grid. To avoid grid congestion and other undesirable effects such as blackouts, charging at moments with an already large static load should be avoided. This objective can be modeled in several ways (see also Section~5.2 of \cite{SchootUiterkamp2022}). However, due to technical constraints on EV batteries, EVs cannot charge at rates very close to zero, i.e., either they are idle (charge at rate zero) or can charge within a particular range of rates \cite{Young2013}. Moreover, charging at low rates is generally more inefficient and thus should be avoided \cite{ApostolakiIosifidou2017}. 

An important aspect of the DEM paradigm is that device-level problems, such as the minimum-threshold EV charging problem, are solved on local embedded systems located within, e.g., households or the charging equipment. Therefore, the utilized device-level optimization algorithms must be very fast in practice because often they are called multiple times within the corresponding DEM system. Furthermore, the embedded systems on which the algorithms run generally have limited computational power \cite{Beaudin2015}. As a consequence, efficient and tailored device-level optimization algorithms for, e.g., the minimum-threshold EV charging problem are crucial ingredients for the real-life implementation of DEM systems.

Disjoint interval bound constraints of the form (\ref{eq_box}) occur also in other applications. For instance, they appear in portfolio optimization problems to model minimum buy-in restrictions (see, e.g., \cite{Jobst2001}). In such problems, investments in assets cannot be arbitrarily small and must be above a given threshold if the investment is actually made. This means that the amount $x_i$ that can be invested in asset $i$ is either zero or lies within a specific interval, leading to constraints of the form $x_i \in \lbrace 0 \rbrace \cup [l_i, u_i]$. If short-selling is allowed, the invested amount can also be negative. This leads to constraints of the form $x_i \in [l'_i, u'_i] \cup \lbrace 0 \rbrace \cup [l_i, u_i]$ with $l'_i < u'_i < 0$, which corresponds to the case $m=3$.

RAP-DIBC also occurs as a subproblem when optimizing non-separable functions over disjoint interval bound constraints using the alternating direction method of multipliers (ADMM). More precisely, in this setting, one of the two iterate updates within the standard ADMM framework requires a projection of the current iterate onto the disjoint interval bound constraints (see, e.g., \cite{Boyd2011}). This subproblem is equivalent to RAP-DIBC when we choose $\phi(x_i + b_i) := (x_i + b_i)^2$. One concrete example of an optimization problem with disjoint interval bound constraints that is successfully solved in this way is given in \cite{Ogarko2021} for inverse problems in geophysics.

%For $m=3$, the disjoint interval constraints \ref{eq_box} can be used to model minimum threshold or minimum buy-in restrictions. Such restrictions are common in portfolio optimization problems and rebalancing with transaction costs REFERENCE. In this setting, adapting an existing investment (e.g., buying or selling assets) is profitable only if this change in investment exceeds a given threshold $\delta$. This leads to constraints of the form $x_i \in (-\infty, x'_i - \delta] \cup \lbrace x'_i \rbrace \cup [x'_i + \delta, \infty)$.

%Moreover, they can model constraints on the absolute value of variables. For instance, the constraint $|x_i| \geq A$ for some $A > 0$ can be modeled as the disjoint interval constraint $x_i \in [-\infty,-A] \cup [A, \infty]$.

%Schur-convexity is defined in terms of majorization. For any vector $x \in \mathbb{R}^n$ and $i \in N$, let $x_{(i)}$ denote the $i^{\text{th}}$ largest element of $x$. Given two vectors $x,y \in \mathbb{R}^n$ with $\sum_{i \in N} x_i = \sum_{i \in N} y_i$, we say that $x$ is majorized by $y$ ($x \prec y$) if $\sum_{i=1}^k x_{(i)} \leq \sum_{i=1}^k y_{(i)}$ for all $k \in N$. A function $\Phi$ is Schur-convex if it preserves the order of majorization, i.e., $\Phi(x) \leq \Phi(y)$ whenever $x \prec y$. Notable classes of Schur-convex functions are $p-$norms and symmetric separable convex functions.

For $m = 1$, RAP-DIBC reduces to the simple RAP and the problem can be solved efficiently in polynomial time (see also Section~\ref{sec_QK}). However, already for $m = 2$, RAP-DIBC is NP-hard since the special case with $l_{i,j} = u_{i,j}$ for all $i \in N$ and $j \in M$ reduces from the subset-sum problem. In fact, even if the collection of intervals is the same for each variable, the problem is NP-hard since it reduces from the even/odd partition problem \cite{vanderKlauw2017}. For the special case $m=2$ with $l_{i,1} = u_{i,1} = 0$ , the problem with a weighted linear objective is known in the literature as the knapsack problem with setups \cite{Michel2009} or with semi-continuous variables \cite{Sun2013, deFarias2013}. Solution approaches for this problem exploit the knapsack structure but generally do not consider the computational complexity of the proposed algorithms. One exception is \cite{Diao2017}, who consider an unweighted linear objective with a relaxation of the resource constraint~(\ref{eq_res}) and demonstrate polynomial-time solvability for specific choices of the disjoint interval bound constraints (\ref{eq_box}). Moreover, \cite{SchootUiterkamp2018} developed an efficient algorithm for Min-Thres-EV with $O(n \log n)$ time complexity for the quadratic objective $\phi(x_i + b_i) = (x_i + b_i)^2$.

\subsection{Contributions}

In this paper, we consider a special case of RAP-DIBC where the collection of intervals is the same for each variable except potentially the first and last interval. More precisely, for each $j \in M \backslash \lbrace 1 \rbrace$ we have $l_{i,j} = \tilde{l}_j$ for some $\tilde{l}_j \in \mathbb{R}$ and for each $j \in M \backslash \lbrace m \rbrace$ we have $u_{i,j} = \tilde{u}_j$ for all $i \in N$ for some $\tilde{u}_j \in \mathbb{R}$. With regard to the lengths of the first and last intervals, we distinguish between two cases for each of them: either their lengths are non-increasing or they are at least the maximum distance between two consecutive intervals. This leads to four different combinations of options for their lengths, which we encode as described in Table~\ref{tab_encode}. In the remainder of this paper, whenever we refer to RAP-DIBC, we mean one of these four special cases of the problem, unless stated otherwise.
\begin{table}[ht!]
\centering
\begin{tabular}{l | l l}
\toprule
& \multicolumn{2}{c}{Length of last interval} \\
& & $\min_{i \in N} (u_{i,m} - l_{i,m}) $  \\
Length of first interval &  $u_{s,m} \geq u_{t,m}$ whenever $s < t$ & $\quad \geq \max_{j \in M \backslash \lbrace m \rbrace} (\tilde{l}_{j+1} - \tilde{u}_j)$ \\
\midrule
$l_{s,1} \leq l_{t,1}$ whenever $s < t$ & (F1,L1) & (F1,L2) \\
$ \min_{i \in N} (u_{i,1} - l_{i,1}) $  \\
$\quad \geq \max_{j \in M \backslash \lbrace m \rbrace} (\tilde{l}_{j+1} - \tilde{u}_j)$ & (F2,L1) & (F2,L2) \\
\bottomrule
\end{tabular}
\caption{Encoding for the considered options of the lengths of the first and last intervals.}
\label{tab_encode}
\end{table}

A naive approach to solve RAP-DIBC would be to consider each possible combination of intervals for the variables separately and solve the corresponding simple RAP. Since one instance of this simple RAP can be solved in $O(n)$ time \cite{SchootUiterkamp2022}, this approach has a non-polynomial time complexity of $O(m^n n F)$, where $F$ denotes the number of flops required for one evaluation of the function $\phi$. Instead, in this paper, we propose an algorithm for solving all the four special cases (F1,L1), (F1,L2), (F2,L1), and (F2,L2) that runs in $O \left(\binom{n+m-2}{m-2} (n \log n + nF) \right)$ time. Note that this complexity is polynomial in $n$ for fixed $m$. We also consider the restriction of the problem to integer variables, as is common in the RAP literature \cite{Ibaraki1988, Hochbaum1994}. We show that only minimal adjustments to the original algorithm are necessary to have it also output an optimal solution to the integral problem. This adjustment does not change the worst-case time complexity of the algorithm.

Our approach is based on two core properties of the problem. We first show that there exists an optimal solution with a specific monotonicity property regarding the used intervals. More precisely, there exists an optimal solution $x^*$ such that, given $i \in N$, we have that $x^*_i \in [l_{i,j}, u_{i,j}]$ implies that $x^*{i'} > u_{i,j}$ for any $i' > i$. As a consequence, compared to the naive approach, we only need to consider combinations of intervals that satisfy this property, which is $O \left(\binom{n + m -1}{m - 1} \right)$.  Secondly, we demonstrate that particular sequences of instances of the simple RAPs corresponding to these combinations can be solved with the same time complexity as solving one instance. For this, we exploit a known monotonicity property of optimal solutions to simple RAPs and the properties of the so-called sequential breakpoint search approach to solve them \cite{Kiwiel2008,Patriksson2015} (see also Section~\ref{sec_QK}). We exploit these properties to show that it is not necessary to solve each instance in the sequence from scratch. Instead, we obtain the input parameters for the next instance in the sequence from the optimal solution and the bookkeeping parameters of the sequential breakpoint search for the previous instance. This can be done efficiently in $O(1)$ time per instance and solving an instance from scratch takes at least $O(n)$ time \cite{SchootUiterkamp2022}. Thus, the overall time complexity of solving the sequence of instances reduces from $O(n^2)$ to $O(n \log n)$, i.e., to the time complexity of the sequential breakpoint search approach for solving a single instance of the simple RAP.

Our approach also provides a partial answer to an open question in \cite{SchootUiterkamp2023}. The simple RAP is known to have a nice reduction property, namely that there exists a solution to the problem that is simultaneously optimal for any choice of continuous convex function $\phi$. The open question posed in the conclusions of \cite{SchootUiterkamp2023} asks whether this property also holds for variants of the simple RAP, in particular for problems with semi-continuous variables. We demonstrate that this is unfortunately not the case by constructing an instance of RAP-DIBC as a counterexample. However, we do show that the reduction property can in fact be used to speed up and simplify several parts of our approach.

We evaluate the performance of our algorithm on realistic instances of Min-Thres-EV and of a set of synthetically generated instances with quadratic objective functions, varying sizes of $n$ and some small values of $m$. Our evaluations demonstrate the suitability of our algorithm for integration in DEM systems due to its small execution time. Furthermore, when evaluated on the synthetical instances, our algorithm outperforms the general-purpose solver Gurobi for most considered instances by up tot two orders of magnitude.

Summarizing, our contributions are as follows:
\begin{enumerate}
\item We introduce the symmetric separable convex RAP with disjoint interval bound constraints that has applications in, among others, electric vehicle charging with minimum charging thresholds;
\item We present an algorithm for this problem whose worst-case time complexity is polynomial in the number of variables, provide the number of intervals $m$ is fixed;
\item We show that the integer version of the problem can be solved using the same algorithm with only a minor adjustment that does not affect the worst-case time complexity;
\item We demonstrate the scalability of our approach for small numbers of disjoint intervals $m$ and its suitability for solving problems arising in DEM systems.
\end{enumerate}

The outline of the remainder of this paper is as follows. In Section~\ref{sec_prob}, we study the feasibility of RAP-DIBC and derive the crucial monotonicity property of optimal solutions to RAP-DIBC. In Section~\ref{sec_alg_init}, we use this property to derive an initial solution approach and algorithm for RAP-DIBC and in Section~\ref{sec_alg_final}, we present an improvement of this algorithm with an $O(\log n)$ time complexity gain. Section~\ref{sec_integer} discusses the extension of our approach and algorithm to the integer-valued version of RAP-DIBC. We present our computational results in Section~\ref{sec_eval} and, finally, present our conclusions in Section~\ref{sec_concl}.

\section{Problem analysis}
\label{sec_prob}

In this section, we establish two important properties of RAP-DIBC. First, in Section~\ref{sec_feasible}, we investigate the feasibility of RAP-DIBC and identify several sufficient conditions for feasibility that can be checked in polynomial time. Second, in Section~\ref{sec_opt}, we establish the existence of a monotone optimal solution to RAP-DIBC. The latter property is the crucial ingredient for our initial solution approach in Section~\ref{sec_alg_init}.

\subsection{Feasibility}
\label{sec_feasible}

We consider the complexity of finding a feasible solution to RAP-DIBC in each of the four cases (F1,L1), (F1,L2), (F2,L1), and (F2,L2). The main result of this section, Lemma~\ref{lemma_feasible}, demonstrates that, while finding a feasible solution for the special case (F1,L1) remains NP-complete, a feasible solution to the other three cases can be found in polynomial time by means of a greedy procedure. This suggests that these cases are simpler than (F1,L1).

\begin{lemma}
Deciding on the existence of a feasible solution for RAP-DIBC to the special case (F1,L1) is NP-complete. For the other three cases, a feasible solution can be found in polynomial time if
\begin{description}
\item[(F1,L2)] $R > n \tilde{u}_1$ or $R < \tilde{l}_m + \sum_{i < n} l_{i,1}$;
\item[(F2,L1)] $R > \sum_{i < n} u_{i,m} + \tilde{u}_1$ or $R < n \tilde{l}_m$;
\item[(F2,L2)] $\sum_{i \in N} l_{i,1} \leq R \leq \sum_{i \in N} u_{i,m}$.
\end{description}
\label{lemma_feasible}
\end{lemma}
\begin{proof}
As noted before, the problem of finding a feasible solution to the case (F1,L1) reduces from the even/odd partition problem, which is NP-complete \cite{vanderKlauw2017}. For the other three cases, we construct a feasible solution $y'$ to RAP-DIBC under the stated conditions for each case as follows.

We start with a specific initial solution $y$, to be defined below, that only satisfies the resource constraint (\ref{eq_res}) and the smallest lower and largest upper bound on each variable, i.e., that belongs to the set $S := \left\{ x \in \mathbb{R}^n \ \mid \ \sum_{i \in N} x_i = R, \ l_{i,1} \leq x_i \leq u_{i,m} \forall i \in N \right\}$. We construct $y$ via a sequential greedy procedure that, starting from the first variable $y_1$, assigns the maximum possible value to variables $y_i$ until the full resource value has been used:
\begin{align*}
y_1 & := \max_{x \in S} x_1 = \min (u_{1,m}, R); \\
y_i &:= \max_{x \in S, \  x_k = y_k \forall k < i } x_i
= \min \left( u_{i,m}, R - \sum_{k=1}^{i-1} u_{k,m} \right), \quad i > 1.
\end{align*}
Note that by construction at most one variable $y_s$ of $y$ is not in one of the disjoint intervals $[l_{s,j},u_{s,j}]$, namely the one whose index $s$ satisfies $\sum_{i=1}^{s-1} u_{i,m} + \sum_{i=s}^n l_{i,1} < R < \sum_{i=1}^{s} u_{i,m} + \sum_{i=s+1}^n l_{i,1}$. Let $j$ be the largest index in $M$ so that $u_{s,j} < y_s < l_{s, j+1}$. We consider each of the three remaining cases (F1,L2), (F2,L1), and (F2,L2) separately:
\begin{description}
\item[(F1,L2)] If $R \leq n \tilde{u}_1$, then the solution $y'$ given by $y'_t := l_{t,1} + (u_{t,1} - l_{t,1}) \frac{R - \sum_{i \in N} l_{i,1}}{\sum_{i \in N} (u_{i,1} - l_{i,1})}$ is feasible for (F1,L2) (and (F2,L2)). If $R \geq \tilde{l}_m + \sum_{i > 1} l_{i,1}$ and $y$ is infeasible, then $s > 1$. It follows that the solution $y'$ with $y'_s = y_s + (l_{s,j + 1} - y_s) = l_{s,j + 1}$, $y'_1 = y_1 - (l_{s,j + 1} - y_s) = u_{1,m} - (l_{s,j + 1} - y_s)$, and $y'_i = y_i$ for $i \not\in \lbrace 1,s \rbrace$ is feasible for (F1,L2) (and (F2,L2)) since $y'_1 \geq u_{1,m} - (l_{s,j+1} - u_{s,j}) \geq u_{1,m} - (u_{1,m} - l_{1,m}) = l_{1,m}$.

\item[(F2,L1)] If $R > n \tilde{l}_m$, then the solution $y'$ given by $y'_t := \tilde{m} + (u_{t,m} - l_{t,m}) \frac{R - \sum_{i \in N} l_{i,m}}{\sum_{i \in N} (u_{i,m} - l_{i,m})}$ is feasible for (F2,L1) (and (F2,L2)). If $R \leq \sum_{i < n} u_{i,m} + \tilde{u}_1$ and $y$ is infeasible, then $s < n$. It follows that the solution $y'$ with $y'_s = y_s - (y_s - u_{s,j}) = u_{s,j}$, $y'_n = y_n + (y_s - u_{s,j}) = l_{n,1} + (y_s - u_{s,j})$, and $y'_i = y_i$ for $i \not\in \lbrace s,n \rbrace$ is feasible for (F2,L1) (and (F2,L2)) since $y'_n \leq l_{n,1} + (l_{s,j+1} - u_{s,j}) \leq l_{n,1} + (u_{n,1} - l_{n,1}) = u_{n,1}$.

\item[(F2,L2)] Note that $\sum_{i \in N} l_{i,1} \leq R \leq \sum_{i \in N} u_{i,1}$ implies that $R \geq \tilde{l}_m + \sum_{i > 1} l_{i,1}$ or  $R < \tilde{l}_m + \sum_{i > 1} l_{i,1} \leq u_{1,m} + \sum_{1 < i < n} u_{i,m} + u_{n,1} = \sum_{i < n} u_{i,m} + \tilde{u}_1$. In the parts for (F1,L2) and (F2,L1), the existence of feasible solutions for (F2,L2) was shown for both these possibilities for values of $R$. Thus, a feasible solution to the case (F2,L2) can be found in polynomial time.
\end{description}
\end{proof}

Note that not all possible parameter settings are covered by Lemma~\ref{lemma_feasible}, meaning that for such cases checking the feasibility of RAP-DIBC might be NP-complete. However, we will show in Sections~\ref{sec_opt} and~\ref{sec_alg_init} that a given instance is infeasible if our eventual solution algorithm terminates without an optimal solution when applied to this instance. Thus, the worst-case time complexity of this Algorithm~\ref{alg_final}, $O \left( \binom{n + m - 2}{m - 2}  (n \log n + nF) \right)$ (see Theorem~\ref{th_complexity}), is also an upper bound on the time required to check feasibility, meaning that for fixed $m$ checking feasibility can be done in polynomial time.

\subsection{Existence of a monotone optimal solution}
\label{sec_opt}

We now focus on deriving the existence of a monotone optimal solution to RAP-DIBC with regard to the used intervals $M$. For a given feasible solution $x$ to RAP-DIBC and an index $i \in N$, let $j(x,i)$ denote the interval that contains $x_i$, i.e., $j(x_i) = j$ if and only if $l_{i,j} \leq x_i \leq u_{i,j}$. We show in Lemma~\ref{lemma_opt} that there exists an optimal solution $x^*$ for which the sequence $(j(x^*,i))_{i \in N}$ is non-decreasing.

\begin{lemma} 
For any feasible instance of RAP-DIBC, there exists an optimal solution $x^*$ so that $j(x^*,i) \leq j(x^*,j)$ whenever $i < j$.
\label{lemma_opt}
\end{lemma}
\begin{proof}
Let $x^*$ be any optimal solution for RAP-DIBC and suppose that there exist indices $s,t$ such that $s < t$ but $j(x^*,s) > j(x^*,t)$. Consider the solution $y(\varepsilon)$ with $y_s(\varepsilon) = x_s^* - \varepsilon$, $y_t(\varepsilon) = x_t^* + \varepsilon$, and $y_i(\varepsilon) = x_i^*$ for $i \not\in \lbrace s,t \rbrace$. Since $x_s^* + b_s >  x_t^* + b_t$, we have for $\varepsilon \in (0, x^*_s + b_s - x^*_t - b_t)$ by convexity of $\phi$ that $\phi(y_s(\varepsilon) + b_s) + \phi(y_t(\varepsilon) + b_t) \leq \phi(x_s^* + b_s) + \phi(x_t^* + b_t)$ and thus $\sum_{i \in N} \phi(y_i(\varepsilon) + b_i) < \sum_{i \in N} \phi(x^*_i + b_i)$. This means that for such $\varepsilon$, the solution $y(\varepsilon)$ is optimal if it is feasible. The argument in this paragraph can then be applied inductively to $y(\varepsilon)$ to arrive at the existence of an optimal solution to RAP-DIBC that satisfies the result of the lemma.

It remains to be shown that $y(\varepsilon)$ is feasible for some $\varepsilon \in (0, x^*_s + b_s - x^*_t - b_t)$. We distinguish between the following cases:
\begin{enumerate}
\item
When $l_{t,1} \leq x^*_s \leq u_{t,m}$ and $l_{s,1} \leq x^*_t \leq u_{s,m}$, feasibility is achieved for $\epsilon = x^*_s - x^*_t$, i.e., by interchanging the variable values between $s$ and $t$.
\item
If $x^*_s < l_{t,1}$, then by assumption we have $x^*_s < l_{t,1} \leq x^*_t < x^*_s$, which is a contradiction. Thus, this case cannot occur. Similarly, the case $x^*_t > u_{s,m}$ cannot occur.
\item
The case $x^*_s > u_{t,m}$ can occur only in the special cases (F1,L2) and (F2,L2), meaning that we may assume that $\min_{i \in N} (u_{i,m} - l_{i,m}) \geq \max_{j \in M \backslash \lbrace m \rbrace} (\tilde{l}_{j+1} - \tilde{u}_j)$. We now consider two cases. If $x^*_t \neq \tilde{u}_{j'}$ for any $j' \in M \backslash \lbrace m \rbrace$, then a sufficiently small $\varepsilon$ suffices. Otherwise, if $x^*_t = \tilde{u}_{j'}$ for some $j' \in M \backslash \lbrace m \rbrace$, the choice $\varepsilon = \varepsilon' := \tilde{l}_{j'+1} - \tilde{u}_{j'}$ suffices since $0 < \varepsilon' \leq u_{t,m} - x^*_t < x^*_s - x^*_t$ and we have $y_t( \varepsilon') = x^*_t +  \tilde{l}_{j'+1} - \tilde{u}_{j'} = \tilde{l}_{j'+1}$ and
\begin{align*}
y_s ( \varepsilon') &= x^*_s - ( \tilde{l}_{j'+1} - \tilde{u}_{j'}) > u_{t,m}  - \max_{j \in M \backslash \lbrace m \rbrace} (\tilde{l}_{j+1} - \tilde{u}_{j}) \geq u_{t,m} - \min_{i \in N} (u_{i,m} - l_{i,m}) \\
&  \geq u_{t,m} - (u_{t,m} - l_{t,m}) = l_{t,m} = \tilde{l}_m,
\end{align*}
meaning that $y(\varepsilon')$ is feasible.

\item
In the final case $x^*_t < l_{s,1}$, feasibility is achieved analogously to the case $x^*_s > u_{t,m}$ by noting that this case occurs only in the special cases (F2,L1) and (F2,L2).
\end{enumerate}
We conclude that there always exists $\varepsilon \in (0, x^*_s + b_s - x^*_t - b_t)$ such that $y(\varepsilon)$ is feasible, which completes the proof of the lemma.
\end{proof}

\section{An initial algorithm}
\label{sec_alg_init}

Based on the existence of a monotone optimal solution to RAP-DIBC as proven in Lemma~\ref{lemma_opt} in the previous section, we present in this section an initial algorithm for solving RAP-DIBC. Lemma~\ref{lemma_opt} implies that there exists an optimal solution $x^*$ and a vector $K^* \in \mathbb{Z}^{m-1}$ with $0 \leq K^*_1 \leq \ldots \leq K^*_{m-1} \leq n$ such that
\begin{align*}
i \leq K^*_1 & \Leftrightarrow x^*_i \in [l_{i,1}, \tilde{u}_1], \\
K^*_{j-1} < i \leq K^*_{j} & \Leftrightarrow x^*_i \in [\tilde{l}_j, \tilde{u}_j], \quad j \in M \backslash \lbrace 1, m \rbrace; \\
K^*_{m-1} < i & \Leftrightarrow x^*_i \in [\tilde{l}_m , u_{i,m}].
\end{align*}
We call $K^*$ an optimal \emph{partition vector} and denote by $\mathcal{K}$ the collection of all valid, i.e., non-decreasing partition vectors, meaning that $\mathcal{K} := \lbrace K \in \mathbb{Z}^{m-1} \ | \ 0 \leq K_1 \leq \ldots \leq K_{m-1} \leq n \rbrace$. For each $K \in \mathcal{K}$ and $i \in N$, we define $j^K(i)$ as the index $j \in M$ such that $K_{j - 1} < i \leq K_{j}$. For each $K \in \mathcal{K}$, we define $P(\phi,K)$ as the restriction of the original problem RAP-DIBC wherein each variable must lie in the interval that is specified by the partition induced by $K$:
\begin{align*}
P(\phi,K) \colon \ \min_{x \in \mathbb{R}^n} \ & \sum_{i \in N} \phi (x_i+b_i) \\
\text{s.t. } & \sum_{i \in N} x_i = R; \\
& x_i \in [l_{i,j^K(i)}, u_{i,j^K(i)}], \quad i \in N.
\end{align*}

The existence of $K^*$ leads to the following general approach to solve RAP-DIBC. For each partition vector $K \in \mathcal{K}$, we compute an optimal solution to $P(\phi,K)$ and record the corresponding optimal objective value $V^K$. Subsequently, we select the partition vector for which $V^{K}$ is the smallest and retrieve the corresponding optimal solution. The question that remains is how to solve $P(\phi,K)$ efficiently for a given $K$. For this, we first explicitly define a special case of $P(\phi,K)$ where $\phi$ is the quadratic function $\phi(x_i + b_i) := \frac{1}{2}(x_i + b_i)^2$:
\begin{align}
Q(K) \colon \ \min_{x \in \mathbb{R}^n} \ & \sum_{i \in N} \frac{1}{2} (x_i + b_i)^2 \nonumber \\
\text{s.t. } & \sum_{i \in N} x_i = R; \label{eq_res_constr} \\
& x_i \in [l_{i,j^K(i)}, u_{i,j^K(i)}], \quad i \in N. \nonumber
\end{align}
Note that both $Q(K)$ and $P(\phi,K)$ are simple RAPs with the same feasible region and shift parameter $b$. As a consequence, we may apply a reduction result from \cite{SchootUiterkamp2022}, which states that optimal solutions to $Q(K)$ are also optimal for $P(\phi,K)$:
\begin{lemma}[Condition 1 and Theorem 1 in \cite{SchootUiterkamp2022}]
Given $K \in \mathcal{K}$ and a continuous convex function $\phi$ so that $P(\phi,K)$ (and thus also $Q(K)$) is feasible, any optimal solution to $Q(K)$ is also optimal for $P(\phi,K)$.
\label{lemma_reduction}
\end{lemma}
Lemma~\ref{lemma_reduction} implies that solving $P(\phi,K)$ reduces to solving $Q(K)$, which is signifcantly simpler. Many different approaches and algorithms to solve $Q(K)$ exist \cite{Patriksson2008}, of which the most efficient ones have an $O(n)$ worst-case time complexity \cite{Brucker1984, Kiwiel2008}.

Algorithm~\ref{alg_init} summarizes the sketched approach. The worst-case time complexity of this algorithm is established as follows. The number of valid partitions is equal to the number of ways that $m-1$ items can be divided over $n+1$ bins, which is $\left( \begin{array}{c} n + m - 1 \\ m-1 \end{array} \right)$. Moreover, as stated before, each instance of $Q(K)$ can be solved in $O(n)$ time \cite{Brucker1984}. Assuming that one evaluation of $\phi$ takes $F$ flops, we conclude that the worst-case time complexity of Algorithm~\ref{alg_init} is $O\left( \left( \begin{array}{c} n + m - 1 \\ m-1 \end{array} \right) (n + nF) \right)$.

\begin{algorithm}
\caption{Ennumerative algorithm for RAP-DIBC.}
\label{alg_init}
\begin{algorithmic}[5]
\STATE{\textbf{Input:} Continuous convex function $\phi$, resource value $R$,  feasible regions $\cup_{j \in M} [l_{i,j},u_{i,j}]$ for each $i \in N$ satisfying (F1,L1), (F1,L2), (F2,L1), or (F2,L2)}
\STATE{\textbf{Output:} Optimal solution $x^*$ for RAP-DIBC}
\STATE{Establish set of valid partitions: $\mathcal{K} := \lbrace K \in \mathbb{Z}^{m-1} \ | \ 0 \leq K_1 \leq \ldots \leq K_{m-1} \leq n \rbrace$}
\FOR{$K \in \mathcal{K}$}
\IF{$\sum_{i \in N} l_{i,j^K(i)} > R$ or $\sum_{i \in N} u_{i,j^K(i)} < R$}
\STATE{$Q(K)$ has no feasible solution; set $V^K := \infty$}
\ELSE
\STATE{Compute optimal solution $x^K$ to $Q(K)$ and evaluate the optimal objective value for $P(\phi,K)$: $V^K := \sum_{i \in N} \phi(x^K_i + b_i)$}
\ENDIF
\ENDFOR
\IF{$\min_{K \in \mathcal{K}} V^K = \infty$}
\RETURN{Instance is infeasible}
\ELSE
\STATE{Select optimal partition vector $K^* := \arg\min_{K \in \mathcal{K}}  V^K$ and compute optimal solution $x^* := x^K$}
\RETURN{$x^*$}
\ENDIF
\end{algorithmic}
\end{algorithm}

We conclude this section with three remarks. First, in practice, one may consider to use alternative subroutines for solving the simple RAP subproblems that do not achieve the best known worst-case time complexity of $O(n)$. This is because the linear time complexity of the algorithms in, e.g., \cite{Brucker1984, Kiwiel2008} is achieved by using linear-time algorithms for finding medians, which are relatively slow in practice (see also the discussion in Section 4.1 of \cite{SchootUiterkamp2022}). As a consequence, alternative methods are often faster in practice and simpler to implement (e.g., the sequential breakpoint search approach described in Section~\ref{sec_QK}).

Second, the approach and algorithm in this section can be generalized to the case where the objective function of RAP-DIBC is $\Phi(x + b)$ for some Schur-convex function $\Phi \colon \mathbb{R}^n \rightarrow \mathbb{R}$ (see, e.g., \cite{Marshall2011} for more background on such functions). This is because the results of both Lemmas~\ref{lemma_opt} and~\ref{lemma_reduction} also hold for this more general case:
\begin{itemize}
\item
In Lemma~\ref{lemma_opt}, the necessary fact that $\Phi(y(\varepsilon) + b) \leq \Phi(x^* + b)$ for $\varepsilon \in (0, x^*_s + b_s - x^*_t - b_t)$ follows directly from the characterization of Schur-convex functions in, e.g., Lemma~3.A.2 of \cite{Marshall2011};
\item
Lemma~\ref{lemma_reduction} can be extended from continuous convex functions to Schur-convex functions (see, e.g., Theorem~5 of \cite{SchootUiterkamp2023}).
\end{itemize}
As a consequence, the only necessary adaption to Algorithm~\ref{alg_init} is in Line~8, where now the objective value of $x^K$ must be calculated as $V^K := \Phi(x^K + b)$. Assuming that this calculation takes $\tilde{F}$ flops, the worst-case time complexity of the algorithm becomes $O\left( \left( \begin{array}{c} n + m - 1 \\ m-1 \end{array} \right) (n + \tilde{F}) \right)$.

Finally, regarding the reduction result of Lemma~\ref{lemma_reduction}, one may ask whether a similar result holds for RAP-DIBC itself. More precisely, is it true that any optimal solution to a given instance of RAP-DIBC with quadratic objective is also optimal for that instance for any choice of continuous convex function $\phi$? If this were true, it is not necessary to record for each partition vector the corresponding optimal objective value for $P(\phi,K)$ and, instead, it would suffice to compare only the objective values for $Q(K)$. Unfortunately, Lemma~\ref{lemma_reduction} cannot be extended to RAP-DIBC, as demonstrated by the following counterexample. Given $n$ and a value $L > 0$, consider an instance of RAP-DIBC with $m=2$ and the following parameter choices:
\begin{enumerate}
\item $l_{i,1} = u_{i,1} = 0$, $l_{i,2} = L$ and $u_{i,2} = nL$ for all $i \in N$;
\item $R = nL$, $b_1 = - \frac{1}{2} \frac{nL}{n-1} - \varepsilon$ for some $\varepsilon \in \left(0,\frac{L}{n-1}(\frac{1}{2} n -1) \right)$, and $b_i = -\frac{nL}{n-1}$ for all $i > 1$.
\end{enumerate}
Consider the $n$ possible partitions $K^{i} = (i-1)$ for $i \in N$. For each $i \in N$, the (unique) optimal solution to $Q(K^i)$ is to divide the resource value $R = nL$ equally over all the variables $x_i, \ldots, x_n$ that are not fixed to $0$, i.e., the optimal solution is $x^i := \left(\underbrace{ 0,\ldots,0}_{i-1},\underbrace{  \frac{nL}{n-i+1},\ldots,  \frac{nL}{n-i+1}}_{n-i+1} \right)$. The corresponding objective value of $x^1$ is 
\begin{align*}
V_Q^1 & := \left(\frac{nL}{n} - \frac{1}{2} \frac{nL}{n-1} - \varepsilon \right)^2 + (n-1) \left(\frac{nL}{n} - \frac{nL}{n-1} \right)^2 \\
%&= \left(L - \frac{1}{2} \frac{nL}{n-1} - \varepsilon \right)^2 + (n-1) \left(L - \frac{nL}{n-1} \right)^2 \\
&= L^2 + 2L \left(- \frac{1}{2} \frac{nL}{n-1} - \varepsilon \right) +  \left( - \frac{1}{2} \frac{nL}{n-1} - \varepsilon \right)^2
+ (n-1) \left(\frac{-L}{(n-1)} \right)^2 \\
&= L^2 - \frac{nL^2}{n-1} - 2L \varepsilon +  \left( - \frac{1}{2} \frac{nL}{n-1} - \varepsilon \right)^2 + \frac{L^2}{n-1} \\
&=- 2L \varepsilon +  \left( - \frac{1}{2} \frac{nL}{n-1} - \varepsilon \right)^2
\end{align*}
 and that of $x^i$ for $i \in N \backslash \lbrace 1 \rbrace$ is 
\begin{align*}
V_Q^i  &:= \left( - \frac{1}{2} \frac{nL}{n-1} - \varepsilon \right)^2
+ (i-2) \left(-\frac{nL}{n-1} \right)^2
+ (n-i+1) \left(\frac{nL}{n-i+1} - \frac{nL}{n-1} \right)^2.
\end{align*}
Note that $V_Q^2 = \left( - \frac{1}{2} \frac{nL}{n-1} - \varepsilon \right)^2 < V_Q^i$ for all $i >2$. Moreover, $V_Q^1 = - 2L \varepsilon + V_Q^2 < V_Q^2$. It follows that $V_Q^1 < V_Q^i$ for all $i > 1$. Thus, $x^1$ is the unique optimal solution to RAP-DIBC when choosing $\phi(x_i + b_i) = (x_i + b_i)^2$. However, when choosing $\phi(x_i + b_i) = \max(0, x_i + b_i)$, the objective value of $x^1$ is $\max(0, L - \frac{1}{2} \frac{nL}{n-1} - \varepsilon) = \max(0, \frac{L}{n-1}(\frac{1}{2} n -1) - \varepsilon)$ and that of $x^2$ is $\max(0,- \frac{1}{2} \frac{nL}{n-1} - \varepsilon) = 0$. This means that the objective value of $x^2$ is smaller than that of $x^1$ for $\varepsilon < \frac{L}{n-1}(\frac{1}{2} n -1)$ and thus in that case $x^1$ is not optimal. This negative result also partially answers an open question posed in earlier work \cite{SchootUiterkamp2023} that asks if and how the reduction result in Lemma~\ref{lemma_reduction} can be extended to optimization problems with non-convex feasible regions, in particular to extensions of the simple RAP. The constructed counterexample is an instance of the RAP with semi-continuous variables, meaning that already for this simpler extension the reduction result does not apply anymore.

\section{An improved algorithm}
\label{sec_alg_final}

In this section, we present a different algorithm for RAP-DIBC whose computational complexity improves upon that of Algorithm~\ref{alg_init}. The computational gain compared to Algorithm~\ref{alg_init} is obtained by solving multiple RAP subproblems simultaneously with the same efficiency as solving a single subproblem. For this, we first describe in Section~\ref{sec_QK} an efficient approach for solving single instances of $Q(K)$. Second, in Section~\ref{sec_multiple}, we explain how this approach can be implemented to solve multiple instances of $Q(K)$ with the same worst-case time complexity as solving a single instance.

\subsection{A sequential breakpoint search algorithm for $Q(K)$}
\label{sec_QK}

We first explain how an optimal solution to $Q(K)$ can be found. Here, we closely follow and adjust the approach descried in Section~2 of \cite{SchootUiterkamp2021}. We consider the following Lagrangian relaxation of $Q(K)$:
\begin{align*}
Q(K, \lambda) \colon \ \min_{x \in \mathbb{R}^n} \ & \sum_{i \in N} \frac{1}{2} (x_i + b_i)^2 - \lambda \left(\sum_{i \in N} x_i - R \right) \\
\text{s.t. } & x_i \in [l_{i,j^K(i)}, u_{i,j^K(i)}], \quad i \in N,
\end{align*}
where $\lambda$ is the Lagrange multiplier corresponding to the resource constraint (\ref{eq_res_constr}). Given $\lambda$, the optimal solution to $Q(K,\lambda)$ is given by
\begin{equation}
x_i(K,\lambda) := \begin{cases}
l_{i,j^K(i)} & \text{if } \lambda \leq l_{i,j^K(i)} + b_i; \\
\lambda - b_i & \text{if } l_{i,j^K(i)} + b_i \leq \lambda \leq u_{i,j^K(i)} + b_i; \\
u_{i,j^K(i)} & \text{if } \lambda \geq u_{i,j^K(i)} + b_i.
\end{cases}
\label{eq_x}
\end{equation}
Note that each $x_i(K,\lambda)$ is a continuous, non-decreasing, and piecewise linear function of $\lambda$ with two breakpoints. We denote these breakpoints by $\alpha_i := l_{i,j^K(i)} + b_i$ and $\beta_i := u_{i,j^K(i)} + b_i$ and introduce the breakpoint multisets $\mathcal{A} := \lbrace \alpha_i \ | \ i \in N \rbrace$ and $\mathcal{B} := \lbrace \beta_i \ | \ i \in N \rbrace$. Furthermore, we denote the sum of the variables $x_i(K,\lambda)$ by $z(K,\lambda) := \sum_{i \in N} x_i(K,\lambda)$. Note that also $z(K,\lambda)$ is continuous, non-decreasing, and piecewise linear in $\lambda$ and that its breakpoints are those in $\mathcal{A} \cup \mathcal{B}$. 

Since $Q(K)$ is a convex optimization problem, there exists $\lambda^K$ such that $x(K,\lambda^K)$ is feasible for $Q(K)$, i.e., $z(K,\lambda^K) = R$, and thereby also optimal for $Q(K)$ (see, e.g., \cite{Boyd2004}). The goal is to find this $\lambda^K$ and reconstruct the corresponding optimal solution $x^K$ using (\ref{eq_x}). Note that, in general, $\lambda^K$ is not unique: in that case, the approach in this section finds the \emph{smallest} $\lambda^K$ that satisfies $z(\lambda^K) = R$.

To find $\lambda^K$, we first search for two consecutive breakpoints in $\mathcal{A} \cup \mathcal{B}$, say $\delta_1$ and $\delta_2$, such that $\delta_1 \leq \lambda^K < \delta_2$. Since $z(K,\lambda)$ is non-decreasing, this is equivalent to finding consecutive breakpoints $\delta_1$ and $\delta_2$ such that $z(K,\delta_1) \leq R < z(K,\delta_2)$. Most breakpoint search approaches in the literature propose to find $\delta_1$ and $\delta_2$ using a binary search on the breakpoints. However, we choose to employ a sequential search here, meaning that we consider the breakpoints in non-decreasing order until we have found the smallest breakpoint $\delta'$ with $z(K,\delta') \leq R$. Moreover, we use the following bookkeeping parameters to efficiently compute $z(K,\lambda)$:
\begin{equation*}
B(\lambda):= \sum_{i: \ \lambda < \alpha_i} l_{i,j^K(i)} + \sum_{i: \ \lambda > \beta_i} u_{i,j^K(i)};
\quad
F(\lambda) := \sum_{i: \ \alpha_i \leq \lambda \leq \beta_i} b_i;
\quad
N_F(\lambda) := | \lbrace i: \ \alpha_i \leq \lambda \leq \beta_i \rbrace |.
\end{equation*}
Each time a new breakpoint has been considered, we update these parameters according to Table~\ref{tab_BP}.

\begin{table}[ht!]
\centering
\begin{tabular}{r | l l l}
\toprule
Type of $\lambda$ & Update $B(\lambda)$ & Update $F(\lambda)$ & Update $N_F(\lambda)$ \\
\midrule
$\lambda \equiv \alpha_i$ & $B(\lambda) - l_{i,j^K(i)}$ & $F(\lambda) + b_i$ & $N_F(\lambda) + 1$ \\
$\lambda \equiv \beta_i$ & $B(\lambda) + u_{i,j^K(i)}$ & $F(\lambda) - b_i$ & $N_F(\lambda) - 1$ \\
\bottomrule
\end{tabular}
\caption{Updating the bookkeeping parameters throughout the sequential breakpoint search.}
\label{tab_BP}
\end{table}

Secondly, given $\delta_1$ and $\delta_2$, we find $\lambda^K$ as follows. If $z(K,\delta_1) = R$, then $\lambda^K = \delta_1$ and we are done. Otherwise, we know that $\lambda^K$ is not a breakpoint and therefore, by the monotonicity of $x(K,\lambda)$, that $x_i(K,\lambda^K) = l_{i,j^K(i)}$ if and only if $x_i(K,\delta_2) = l_{i,j^K(i)}$ and that $x_i(K,\lambda^K) = u_{i,j^K(i)}$ if and only if $x_i(K,\delta_1) = u_{i,j^K(i)}$. We can thus directly compute $B(\lambda^K)$, $F(\lambda^K)$, and $N_F(\lambda^K)$ as $B(\delta_1)$, $F(\delta_1)$, and $N_F(\delta_1)$, respectively. Note, that
\begin{align*}
R &= z(K,\lambda^K) \\
&= \sum_{i: \ x_i(K,\lambda^K) = l_{i,j^K(i)}} l_{i,j^K(i)}
+  \sum_{i: \ l_{i,j^K(i)} < x_i(K,\lambda^K) < u_{i,j^K(i)}} (\lambda^K - b_i)
+ \sum_{i: \ x_i(K,\lambda^K) = u_{i,j^K(i)}} u_{i,j^K(i)} \\
&=  B(\lambda^K) + N_F(\lambda^K) \lambda^K - F(\lambda^K)
\end{align*}
and thus we have
\begin{equation*}
\lambda^K = \frac{R -  B(\lambda^K)
+ F(\lambda^K) }{ N_F(\lambda^K)}.
\end{equation*}

Algorithm~\ref{alg_Q} summarizes the sketched approach. The breakpoint multisets $\mathcal{A}$ and $\mathcal{B}$ can be stored as sorted lists, meaning that computing the smallest breakpoint $\lambda_i$ in Line~6 takes $O(1)$ time. Thus, each iteration of the algorithm takes $O(1)$ time. This means that the entire breakpoint search procedure in Lines~5-22 takes at most $O(n)$ time since in the worst case all $2n$ breakpoint values in $\mathcal{A} \cup \mathcal{B}$ must be considered. Thus, the overall complexity of Algorithm~\ref{alg_Q} is $O(n \log n)$ due to the initial sorting of $\mathcal{A}$ and $\mathcal{B}$.

\begin{algorithm}[ht!]
\caption{An $O(n \log n)$ breakpoint search algorithm for $Q(K)$.}
\label{alg_Q}
\begin{algorithmic}[5]
\STATE{\textbf{Input:} Partition vector $K \in \mathcal{K}$, parameters $b \in \mathbb{R}^n$, $l_{i,j^K(i)}, u_{i,j^K(i)}$ for each $i \in N$, resource value $R$}
\STATE{\textbf{Output:} Optimal solution $x^K$ to $Q(K)$ and corresponding optimal Lagrange multiplier $\lambda^K$}
\STATE{Compute the breakpoint multisets $\mathcal{A} := \lbrace \alpha_i \ | \ i \in N \rbrace$ and $\mathcal{B} := \lbrace \beta_i \ | \ i \in N \rbrace$}
\STATE{Initialize $B := \sum_{i \in N} l_{i,j^K(i)}$, $F := 0$, and $N_F : = 0$}
\REPEAT
\STATE{Determine smallest breakpoint $\lambda_i := \min(\mathcal{A} \cup \mathcal{B})$}
\IF{$B + N_F \lambda_i - F = R$}
\STATE{$\lambda^K = \lambda_i$; compute $x^K$ as $x(K,\lambda_i)$ using (\ref{eq_x})}
\RETURN{$x^K, \lambda^K$}
\ELSIF{$B + N_F \lambda_i - F > R$}
\STATE{$\lambda^K = \frac{R - B + F}{N_F}$; compute $x^K$ as $x(K,\lambda^K)$ using (\ref{eq_x})}
\RETURN{$x^K, \lambda^K$}
\ELSE
\IF{$\lambda_i$ is lower breakpoint $(\lambda_i = \alpha_i)$}
\STATE{$B := B - l_{i,j^K(i)}$; $F := F + b_i$; $N_F := N_F + 1$}
\STATE{$\mathcal{A} := \mathcal{A} \backslash \lbrace \alpha_i \rbrace$}
\ELSE
\STATE{$B := B + u_{i,j^K(i)}$; $F := F - b_i$; $N_F := N_F - 1$}
\STATE{$\mathcal{B} := \mathcal{B} \backslash \lbrace \beta_i \rbrace$}
\ENDIF
\ENDIF
\UNTIL{multiplier $\lambda^K$ has been found}
\end{algorithmic}
\end{algorithm}

We conclude this section with a remark that is relevant for the next section. From Lemma~\ref{lemma_reduction}, we know that the output $x^K$ of Algorithm~\ref{alg_Q} is also optimal for $P(\phi,K)$ for any choice of continuous convex function $\phi$. To obtain the objective value for $P(\phi,K)$, we could simply evaluate the objective function of $P(\phi,K)$ for $x^K$. However, we can also adjust Algorithm~\ref{alg_Q} slightly so that this objective value can be computed without explicitly computing $x^K$. To this end, we introduce also the following bookkeeping parameter:
\begin{equation*}
V_B (\lambda) :=  \sum_{i: \ \lambda < \alpha_i} \phi(l_{i,j^K(i)} + b_i) + \sum_{i: \ \lambda > \beta_i} \phi(u_{i,j^K(i)} + b_i).
\end{equation*}
Given $\lambda^K$, the optimal objective value of $x^K$ equals $V_B(\lambda^K) + N_F(\lambda^K) \phi(\lambda^K)$. Analogously to $B(\lambda)$, the parameter $V_B(\lambda)$ is also updated whenever a new breakpoint is considered. More precisely, if this is a lower breakpoint, say $\alpha_i$, then we update $V_B(\lambda)$ to $V_B(\lambda) - \phi(l_{i,j^K(i)} + b_i)$. On the other hand, if it is an upper breakpoint $\beta_i$, we update $V_B(\lambda)$ to $V_B(\lambda) + \phi(u_{i,j^K(i)} + b_i)$. Including this feature in the algorithm changes its worst-case time complexity from $O(n \log n)$ to $O(n \log n + nF)$.

\subsection{Solving multiple subproblems in one run}
\label{sec_multiple}

In this subsection, we describe how Algorithm~\ref{alg_Q} can be adopted to solve $Q(K)$ for a particular sequence of partition vectors while maintaining the original $O(n \log n)$ time complexity. For this, we define for a given $K \in \mathcal{K}$ the partition vector $K^+$ obtained from $K$ by increasing $K_{m-1}$ by $1$, i.e., 
\begin{equation}
K^+ := (K_1, \ldots, K_{m-2}, K_{m-1} + 1 ).
\label{eq_increase}
\end{equation}
The crucial ingredient for our approach is Lemma~\ref{lemma_lambda}, which demonstrates a monotone relationship between the optimal Lagrange multipliers $\lambda^{K^+}$ and $\lambda^K$:
\begin{lemma}
Let a partition vector $K \in \mathcal{K}$ be given and let $K^+$ be the partition vector as defined in (\ref{eq_increase}). Then $\lambda^{K^+} > \lambda^{K}$.
\label{lemma_lambda}
\end{lemma}
\begin{proof}
Note that by definition of $x(\cdot,\lambda)$ in (\ref{eq_x}), we have for any $\lambda \in \mathbb{R}$ that $x_i(K^+,\lambda) = x_i(K,\lambda)$ for all $i \neq K_{m-1}$ and $x_{K_{m-1}}(K^+,\lambda) < x_{K_{m-1}}(K,\lambda)$. It follows that
\begin{equation*}
z(K^+,\lambda^K) =
\sum_{i \in N} x_i(K^+,\lambda^K)
<\sum_{i \in N} x_i(K,\lambda^K)
= R
= z(K^+,\lambda^{K^+}).
\end{equation*}
Since $z(K^+,\lambda)$ is non-decreasing, it follows that $\lambda^{K^+} > \lambda^{K}$.
\end{proof}

We now describe how Algorithm~\ref{alg_Q} can be adjusted to solve both $Q(K)$ and $Q(K^+)$ simultaneously in $O(n \log n)$ time. We start by applying the algorithm to $Q(K)$ and record the optimal multiplier $\lambda^K$ and objective value $V^K$. Next, we solve $Q(K^+)$ using the same algorithm to solve $Q(K^+)$, but we use a different multiplier value to start the breakpoint search. More precisely, instead of starting the search at the smallest breakpoint, we start the search at the previous optimal Lagrange multiplier $\lambda^K$. This is a valid value from which to start the search since $\lambda^K < \lambda^{K^+}$ by Lemma~\ref{lemma_lambda}.

Because we are starting from a different multiplier value, we need to calculate the bookkeeping parameters for this particular value. Normally, this requires at least $O(n)$ time and thus would not lead to any efficiency gain. However, note that the problems $Q(K)$ and $Q(K^+)$ have the same set of breakpoints, except for those corresponding to the index $K_{m-1}$. This means that we do not need to compute the bookkeeping parameters and breakpoint sets from scratch. Instead, we may re-use the parameters and sets corresponding to $\lambda^K$ in $Q(K)$ and adjust them only for the change in the breakpoints corresponding to $K_{m-1}$.

We do this adjustment in two steps. First, we remove from the bookkeeping parameters and breakpoint sets all contributions corresponding to $x_{K_{m-1}}(K,\lambda^K)$. More precisely:
\begin{itemize}
\item If $\lambda^K < \alpha_{K_{m-1}}$, we know that $x_{K_{m-1}}(K,\lambda^K) = l_{K_{m-1},m}$. This value must thus be subtracted from the bookkeeping parameter $B(\lambda^K)$ and its contribution to the objective value, $\phi(l_{K_{m-1},m} + b_{K_{m-1}})$, from $V_B(\lambda^K)$. Moreover, neither of the breakpoints $\alpha_{K_{m-1}}$ and $\beta_{K_{m-1}}$ have been considered yet and must thus be removed from the breakpoint sets $\mathcal{A}$ and $\mathcal{B}$, respectively.

\item If $\alpha_{K_{m-1}} \leq \lambda^K < \beta_{K_{m-1}}$, we know that $l_{K_{m-1},m} < x_{K_{m-1}}(K,\lambda^K) < u_{K_{m-1},m}$. Thus, we must subtract its contribution $b_{K_{m-1}}$ from $F(\lambda^K)$ and $1$ from $N_F(\lambda^K)$. Moreover, $\alpha_{K_{m-1}}$ has already been considered, but $\beta_{K_{m-1}}$ not yet, so we must remove $\beta_{K_{m-1}}$ from $\mathcal{B}$.

\item If $\lambda^K \geq \beta_{K_{m-1}}$, we know that $x_{K_{m-1}}(K,\lambda^K) = u_{K_{m-1},m}$. This value must thus be subtracted from $B(\lambda^K)$ and its contribution to the objective value, $\phi(u_{K_{m-1},m} + b_{K_{m-1}})$, from $V_B(\lambda^K)$. Moreover, note that in contrast to the previous two cases, both breakpoints $\alpha_{K_{m-1}}$ and $\beta_{K_{m-1}}$ have already been considered and thus no removal from the breakpoint sets $\mathcal{A}$ and $\mathcal{B}$ is required.
\end{itemize}
Table~\ref{tab_update_1} summarizes this first set of updating rules.

In a second step, we adjust the parameters and sets so that they take into account the new breakpoint values $\alpha^{K^+}_{K_{m-1}} := l_{K_{m-1},m-1} + b_{K_{m-1}}$ and $\beta^{K^+}_{K_{m-1}} := u_{K_{m-1},m-1} + b_{K_{m-1}}$. More precisely, we determine which breakpoints would already have been considered in the search had we initialized the algorithm for the partition $K^+$. Subsequently, we resume the breakpoint search procedure as normal until $\lambda^{K^+}$ has been found. The corresponding updates to the bookkeeping parameters and breakpoint sets are the reverse of those in the first step, i.e., now we \emph{add} the contributions of $x_{K_{m-1}}{K^+,\lambda^K}$ to the relevant parameters and sets instead of subtracting them. Table~\ref{tab_update_2} summarizes this second set of updating rules.

\begin{table}[ht!]
\centering
\resizebox{\columnwidth}{!}{%
\begin{tabular}{r | l l l l l l}
\toprule
\multicolumn{7}{c}{\textbf{First update round: compare $\lambda^K$ to old breakpoints}} \\
 &   $B(\lambda^K)$ &   $F(\lambda^K)$ &   $N_F(\lambda^K)$ & $V_B(\lambda^K)$ & $\mathcal{A}$ & $\mathcal{B}$ \\
\midrule
$\lambda^K < \alpha_{K_{m-1}}$ & $- l_{K_{m-1},m}$ & n.a. & n.a. & $-\phi(l_{K_{m-1},m} + b_{K_{m-1}})$ & Remove $\alpha_{K_{m-1}}$ & Remove $\beta_{K_{m-1}}$ \\
$\alpha_{K_{m-1}} \leq \lambda^K < \beta_{K_{m-1}}$
& n.a. & $- b_{K_{m-1}}$ & $-1$ & n.a. & n.a. & Remove $\beta_{K_{m-1}}$ \\
$\lambda^K \geq \beta_{K_{m-1}}$ 
& $- u_{K_{m-1},m}$ & n.a. & n.a. & $-\phi(u_{K_{m-1},m} + b_{K_{m-1}})$ & n.a. & n.a. \\
\bottomrule
\end{tabular}}
\caption{First set of updating rules for the bookkeeping parameters in Algorithm~\ref{alg_Q} when switching from $K$ to $K^+$.}
\label{tab_update_1}
\end{table}

\begin{table}[ht!]
\centering
\resizebox{\columnwidth}{!}{%
\begin{tabular}{r | l l l l l l}
\toprule
\multicolumn{7}{c}{\textbf{Second update round: compare $\lambda^K$ to new breakpoints}} \\
 &   $B(\lambda^K)$ &   $F(\lambda^K)$ &   $N_F(\lambda^K)$ & $V_B(\lambda^K)$ & $\mathcal{A}$ & $\mathcal{B}$ \\
\midrule
$\lambda^K < \alpha^{K^+}_{K_{m-1}}$ & $ + l_{K_{m-1},m-1}$ & n.a. & n.a. & $+\phi(l_{K_{m-1},m-1}  + b_{K_{m-1}})$ & Add $\alpha^{K^+}_{K_{m-1}}$ & Add $\beta^{K^+}_{K_{m-1}}$ \\
$\alpha^{K^+}_{K_{m-1}} \leq \lambda^K < \beta^{K^+}_{K_{m-1}}$
& n.a. & $+ b_{K_{m-1}}$ & $+1$ & n.a. & n.a. & Add $\beta^{K^+}_{K_{m-1}}$ \\
$\lambda^K \geq \beta^{K^+}_{K_{m-1}}$ 
& $+u_{K_{m-1},m-1}$ & n.a. & n.a. & $+\phi(u_{K_{m-1},m-1} + b_{K_{m-1}})$& n.a. & n.a. \\
\bottomrule
\end{tabular}}
\caption{Second set of updating rules for the bookkeeping parameters in Algorithm~\ref{alg_Q} when switching from $K$ to $K^+$.}
\label{tab_update_2}
\end{table}

The final part of our approach is the following observation. Up until now, we only considered solving two subproblems in one run of the sequential breakpoint search algorithm. However, we may apply the same methodology again to also solve a third subproblem $Q(K^{++})$, where the partition vector $K^{++}$ is obtained from $K^+$ by increasing its last element $K^+_{m-1}$ by $1$. More precisely, after the optimal multiplier $\lambda^{K^+}$ of $Q(K^+)$ has been found, we may again update the bookkeeping parameters and breakpoint sets according to Tables~\ref{tab_update_1} and~\ref{tab_update_2} so that we initialize the breakpoint search for $Q(K^{++})$ at $\lambda^{K^{++}}$. Thus, we apply the same procedure where now the partition vector $K^+$ takes the role of the initial partition vector $K$ and the new vector $K^{++}$ takes the role of $K^+$. In fact, we may apply the procedure repeatedly to eventually obtain the optimal multipliers and objective values of all subproblems corresponding to partition vectors of the form $(K_1, \ldots, K_{m-2}, \tilde{K})$ with $K_{m-2} \leq \tilde{K} \leq n$ in only a \emph{single} run of the sequential breakpoint search procedure. The parameters for the first subproblem $Q((K_1,\ldots,K_{m-2},K_{m-1}))$ are initialized from scratch as in Algorithm~\ref{alg_init} and those of the subsequent subproblems $Q((K_1,\ldots,K_{m-2},\tilde{K}))$ are initialized based on the optimal multiplier $\lambda^{(K_1,\ldots,K_{m-2},\tilde{K}-1)}$ of the previous subproblem $Q((K_1,\ldots, K_{m-2}, \tilde{K}-1))$.

Algorithm~\ref{alg_final} summarizes our approach. We first construct the collection of all valid ``subpartition vectors'' $\mathcal{K}' := \lbrace K \in \mathbb{Z}^{m-2} \ | \ 0 \leq K_1 \leq \ldots \leq K_{m-2} \leq n \rbrace$. Subsequently, for each of subpartition vector $K' := (K_1, \ldots, K_{m-2})$, we carry out a single breakpoint search procedure for all partitions $K \in \mathcal{K}$ of the form $(K_1,\ldots,K_{m-2},\tilde{K})$ with $K_{m-2} \leq \tilde{K} \leq n$ simultaneously. After $\lambda^K$ and $V^K$ have been found for one such partition vector, we apply the update rules in Tables~\ref{tab_update_1} and~\ref{tab_update_2} to initialize the breakpoint search for the next partition. 
\begin{algorithm}[ht!]
\caption{A refined algorithm for RAP-DIBC.}
\label{alg_final}
\begin{multicols}{2}
\begin{algorithmic}[5]
\STATE{\textbf{Input:} continuous convex function $\phi$; parameter $b \in \mathbb{R}^n$, resource value $R$, regions $\cup_{j \in M} [l_{i,j},u_{i,j}]$ for each $i \in N$ satisfying (F1,L1), (F1,L2), (F2,L1), or (F2,L2)}
\STATE{\textbf{Output:} Optimal solution $x^*$ for RAP-DIBC}
\STATE{Establish set of valid partitions  of size $m-2$: $\mathcal{K}' := \lbrace K \in \mathbb{Z}^{m-2} \ | \ 0 \leq K_1 \leq \ldots \leq K_{m-2} \leq n \rbrace$.}
\FOR{$K' \in \mathcal{K}'$}
\STATE{$K_{m-1} := K_{m-2}$; $K = (K',K_{m-1})$}
\STATE{Compute the breakpoints $\alpha_i := l_{i,j^K(i)} + b_i$ and $\beta_i := u_{i,j^K(i)} + b_i$ for each $i \in N$ and establish $\mathcal{A} := \lbrace \alpha_i \ | \ i \in N \rbrace$ and $\mathcal{B} := \lbrace \beta_i \ | \ i \in N \rbrace$}
\STATE{Initialize $B := \sum_{i \in N} l_{i,j^K(i)}$, $F := 0$, $N_F : = 0$, and $V_B := \sum_{i \in N} \phi(l_{i,j^K(i)} + b_i)$}
\WHILE{$K_{m-1} \leq n$}
\STATE{$K = (K',K_{m-1})$}
\REPEAT
\STATE{Determine smallest breakpoint $\lambda_i := \min(\mathcal{A} \cup \mathcal{B})$}
\IF{$B + N_F \lambda_i - F = R$}
\STATE{$\lambda^K = \lambda_i$; $V^K := V_B + N_F \phi(\lambda^K)$}
\ELSIF{$B + N_F \lambda_i - F > R$}
\STATE{$\lambda^K = \frac{R - B + F}{N_F}$; $ V^K := V_B + N_F \phi(\lambda^K)$}
\ELSE
\IF{$\lambda_i$ is lower breakpoint $(\lambda_i = \alpha_i)$}
\STATE{$B := B - l_{i,j^K(i)}$; $F := F + b_i$; $N_F := N_F + 1$; $V_B := V_B - \phi(l_{i,j^K(i)} + b_i)$}
\STATE{$\mathcal{A} := \mathcal{A} \backslash \lbrace \alpha_i \rbrace$}
\ELSE
\STATE{$B := B + u_{i,j^K(i)}$; $F := F - b_i$; $N_F := N_F - 1$;  $V_B := V_B + \phi(u_{i,j^K(i)} + b_i)$}
\STATE{$\mathcal{B} := \mathcal{B} \backslash \lbrace \beta_i \rbrace$}
\ENDIF
\ENDIF
\UNTIL{multiplier $\lambda^K$ has been found}
\IF{$K_{m-1} = n$}
\STATE{\textbf{break}}
\ELSE
\STATE{$K_{m-1} := K_{m-1} + 1$}
\IF{$\lambda^K < \alpha_{K_{m-1}}$}
\STATE{$B := B - l_{K_{m-1},m}$; $V_B := V_B - \phi(l_{K_{m-1},m} + b_{K_{m-1}})$}
\STATE{Remove $\alpha_{K_{m-1}}$ from $\mathcal{A}$; remove $\beta_{K_{m-1}}$ from $\mathcal{B}$}
\ELSIF{$\alpha_{K_{m-1}} \leq \lambda^K < \beta_{K_{m-1}}$}
\STATE{$F := F - b_{K_{m-1}}$; $N_F := N_F - 1$}
\STATE{Remove $\beta_{K_{m-1}}$ from $\mathcal{B}$}
\ELSE
\STATE{$B := B - u_{K_{m-1},m}$; $V_B := V_B - \phi(u_{K_{m-1},m} + b_{K_{m-1}})$}
\ENDIF
\STATE{Compute new breakpoints $\alpha_{K_{m-1}} = l_{K_{m-1},m-1} + b_{K_{m-1}}$ and $\beta_{K_{m-1}} = u_{K_{m-1},m-1} + b_{K_{m-1}}$ }
\IF{$\lambda^K < \alpha_{K_{m-1}}$}
\STATE{$B := B + l_{K_{m-1},m-1}$; $V_B := V_B + \phi(l_{K_{m-1},m-1} + b_{K_{m-1}})$}
\STATE{Add $\alpha_{K_{m-1}}$ to $\mathcal{A}$; add $\beta_{K_{m-1}}$ to $\mathcal{B}$}
\ELSIF{$\alpha_{K_{m-1}} \leq \lambda^K < \beta_{K_{m-1}}$}
\STATE{$F := F + b_{K_{m-1}}$; $N_F := N_F + 1$}
\STATE{Add $\beta_{K_{m-1}}$ to $\mathcal{B}$}
\ELSE
\STATE{$B := B + u_{K_{m-1},m-1}$; $V_B := V_B + \phi(u_{K_{m-1},m-1} + b_{K_{m-1}})$}
\ENDIF
\ENDIF
\ENDWHILE
\ENDFOR
\IF{$\min_{K \in \mathcal{K}} V^K = \infty$}
\RETURN{Instance is infeasible}
\ELSE
\STATE{Select optimal partition vector $K^* := \arg \min_{K \in \mathcal{K}} V^K$ and reconstruct optimal solution as $x^* := x(K^*, \lambda^{K^*})$}
\RETURN{$x^*$}
\ENDIF
\end{algorithmic}
\end{multicols}
\end{algorithm}

Theorem~\ref{th_complexity} establishes the worst-case time complexity of the algorithm:
\begin{theorem}
An instance of RAP-DIBC satisfying one of the four cases (F1,L1), (F1,L2), (F2,L1), or (F2,L2) can be solved by Algorithm~\ref{alg_final} in $O \left( \binom{n + m - 2}{m - 2}  (n \log n + nF) \right)$ time.
\label{th_complexity}
\end{theorem}
\begin{proof}
We first consider the time complexity of each iteration of the for-loop in Algorithm~\ref{alg_final} and focus on the breakpoint search and parameter updating procedure separately. First, note that throughout one complete breakpoint search, i.e., one iteration of the for-loop in Algorithm~\ref{alg_final}, the two breakpoints for each given variable are each considered at most once (either before or after they have been updated).  Thus, the total number of iterations of the while-loop within one iteration of the for-loop is at most $O(n)$, leading to an overall worst-case time complexity of $O(n \log n)$. In other words, each iteration of the while-loop has $O(\log n)$ amortized time complexity.

Second, with regard to the updating procedure, note that all parameter updates can be done in $O(F)$ time. Moreover, when storing $\mathcal{A}$ and $\mathcal{B}$ as priority queues, removing considered breakpoints (Lines 19 and 22) and adding new breakpoint values (Lines 42 and 45) can be done in $O(\log n )$ time. Removing arbitrary breakpoints (Lines 32 and 35) from such a priority queue would normally take $O(n)$ time. However, we propose a different approach. Instead of removing an outdated breakpoint value, we keep it in the heap. Whenever a new smallest breakpoint is determined in Line~11, we first check if this is such an outdated breakpoint. We do this by checking its value with the current actual breakpoint value. If this does not match, the breakpoint will be removed, which now takes $O(\log n)$ time since it is the smallest breakpoint in the priority queue. This check takes $O(1)$ time, meaning that each updating procedure takes $O(\log n)$ time in total. Thus, the worst-case time complexity of all updating procedures within one iteration of the for-loop of the algorithm is $O(n \log n + nF)$.

It follows that each iteration of the for-loop can be executed in $O(n \log n + nF)$ time. The result of the theorem follows since the number of iterations of the for-loop is $\binom{n + m - 2}{m - 2}$.
\end{proof}

For $m = 2$, the worst-case time complexity in Theorem~\ref{th_complexity} reduces to $O(n \log n + nF)$. This is a significant improvement over the $O(n^2(1+F))$ complexity of Algorithm~\ref{alg_init} and matches the time complexity of the method in \cite{SchootUiterkamp2018} that solves a specific special case with quadratic objective functions. For $m=3$, the complexity of Algorithm~\ref{alg_final} becomes $O(n^2 \log n + n^2F)$, which also improves upon the complexity of Algorithm~\ref{alg_init} by a factor $O \left(\frac{n}{\log n} \right)$.

\section{Integer variables}
\label{sec_integer}

In this section, we present an adjustment to Algorithm \ref{alg_final} so that it also outputs an optimal solution to RAP-DIBC with integer variables. For convenience, we state this problem explicitly as $\tilde{P}$:
\begin{align*}
\tilde{P} \colon \ \min_{x \in \mathbb{Z}^n} \ & \sum_{i \in N} \phi(x_i + b_i)  \\
\text{s.t. } & \sum_{i \in N} x_i = R,  \\
& x_i \in \cup_{j \in M} [l_{i,j} , u_{i,j}], \quad i \in N.
\end{align*}
Here, we assume that all parameters $b$, $R$, $l$, and $u$ are integer-valued.

First, we note that all arguments and lemmas in Section~\ref{sec_prob} to arrive at Algorithm~\ref{alg_init} are also valid for $\tilde{P}$. The only adjustment of this algorithm is in Line~8 where an optimal solution to $Q(K)$ is computed. For $\tilde{P}$, now in each iteration an optimal solution to the integral version $\tilde{Q}(K)$ of $Q(K)$ must be computed. This integral version $\tilde{Q}(K)$ is given by
\begin{align*}
\tilde{Q}(K) \colon \ \min_{x \in \mathbb{Z}^n} \ & \sum_{i \in N} \frac{1}{2} (x_i + b_i)^2 \\
\text{s.t. } & \sum_{i \in N} x_i = R;  \\
& x_i \in [l_{i,j^K(i)}, u_{i,j^K(i)}], \quad i \in N. 
\end{align*}
In general, this means that we can solve $\tilde{P}$ by solving each subproblem $\tilde{Q}(K)$ for all valid partition vectors $K$ and calculate the corresponding objective values for the separable objective function $\sum_{i \in N} \phi(x_i + b_i)$.
Since we require only the optimal objective value and not the optimal solution itself when considering different partition vectors, we propose an approach to find the optimal objective value of $\tilde{Q}(K)$ given the optimal Lagrange multiplier $\lambda^K$ of the original continuous problem $Q(K)$. The advantage of this approach is that it can be easily integrated into Algorithm~\ref{alg_final} and executed each time an optimal multiplier for the continuous version of the problem has been found in Lines~10-25.

Our approach is in its essence a tailored version of the approach in \cite{Weinstein1973} for general separable convex objective functions to $\tilde{Q}(K)$ (see also \cite{Ibaraki1988}). We first consider the following Lagrangian relaxation of $\tilde{Q}(K)$:
\begin{align*}
\tilde{Q}(K, \lambda): \quad \min_{x \in \mathbb{Z}^n} \ & \sum_{i \in N} \frac{1}{2} (x_i + b_i)^2 - \lambda \left(\sum_{i \in N} x_i - R \right) \\
\text{s.t. } & x_i \in [l_{i,j^K(i)}, u_{i,j^K(i)}], \quad i \in N.
\end{align*}
As opposed to the Lagrangian relaxation of $Q(K)$, the current relaxation $\tilde{Q}(K, \lambda)$ does not have a unique solution for every $\lambda$. However, observe that the following conditions are both necessary and sufficient for a feasible solution $x$ to $\tilde{Q}(K,\lambda)$ to also be optimal for $\tilde{Q}(K,\lambda)$:
\begin{equation}
x_i \begin{cases}
= l_{i,j^K(i)} & \text{if } \lambda \leq l_{i,j^K(i)} + b_i; \\
= \lfloor \lambda \rfloor - b_i & \text{if } \ l_{i,j^K(i)} + b_i \leq \lambda \leq u_{i,j^K(i)} + b_i \text{ and } \lambda - \lfloor \lambda \rfloor < \frac{1}{2}; \\
\in \lbrace \lfloor \lambda \rfloor - b_i ,  \lceil \lambda \rceil - b_i \rbrace & \text{if } \ l_{i,j^K(i)} + b_i \leq \lambda \leq u_{i,j^K(i)} + b_i \text{ and } \lambda - \lfloor \lambda \rfloor = \frac{1}{2}; \\
= \lceil \lambda \rceil - b_i & \text{if } \ l_{i,j^K(i)} + b_i \leq \lambda \leq u_{i,j^K(i)} + b_i \text{ and } \lambda - \lfloor \lambda \rfloor > \frac{1}{2}; \\
= u_{i,j^K(i)} & \text{if } \lambda \geq u_{i,j^K(i)} + b_i; \\
\end{cases}
\label{eq_L_opt_Z}
\end{equation}
An important observation is that given $\lambda \in \mathbb{R}$, any optimal solution $\tilde{x}$ to $\tilde{Q}(K,\lambda)$ with $\sum_{i \in N} \tilde{x} = R$ is also optimal for $\tilde{Q}(K)$. To see this, note that for any feasible solution $x$ of $\tilde{Q}(K)$ we have
\begin{align*}
\sum_{i \in N} \frac{1}{2} (\tilde{x}_i + b_i)^2
&=
\sum_{i \in N} \frac{1}{2} (\tilde{x}_i + b_i)^2 - \lambda \left(\sum_{i \in N} \tilde{x}_i - R \right) \\
&\leq
\sum_{i \in N} \frac{1}{2} (x_i + b_i)^2 - \lambda \left(\sum_{i \in N} x_i - R \right)
=
\sum_{i \in N} \frac{1}{2} (x_i + b_i)^2.
\end{align*}
It follows that $\tilde{x}$ is also optimal for $\tilde{Q}(K)$. One way to find such a solution $\tilde{x}$ is to redistribute the fractional parts of the non-integer solution $x^K$ to $Q(K)$ over the active variables, i.e., the variables that are not equal to one of their bounds. Given $\lambda^K$, the sum of these fractional parts equals $N_F^+ := N_F(\lambda^K) (\lambda^K - \lfloor \lambda^K \rfloor)$. Moreover, let $i'$ denote the $N_F^+$-th index for which $x^K_i + b_i = \lambda^K$, i.e., the $N_F^+$-th active variable. We consider the following candidate solution $\tilde{x}$, which is obtained from $x^K$ by redistributing a value of $N_F^+$ as equally as possible over the first $N_F^+$ active variables:
\begin{equation}
\tilde{x}_i := \begin{cases}
x^K_i & \text{if } x^K_i \in \lbrace l_{i,j^K(i)}, u_{i,j^K(i)} \rbrace; \\
\lceil x^K_i \rceil & \text{if } x^K_i + b_i = \lambda^K \text{ and } i \leq i'; \\
\lfloor x^K_i \rfloor & \text{if } x^K_i + b_i = \lambda^K \text{ and } i > i'.
\end{cases}
\label{eq_x_cand}
\end{equation}

We show that $\tilde{x}$ is both feasible and optimal for $\tilde{Q}(K)$. Regarding feasibility, note that for any $i \in N$ with $x^K_i + b_i = \lambda^K$, we have $\lfloor x^K_i \rfloor + b_i = \lfloor \lambda^K \rfloor$ and $\lceil x^K_i \rceil + b_i = \lceil \lambda^K \rceil$. It follows that
\begin{align*}
\sum_{i: \ x^K_i + b_i = \lambda^K, \ i \leq i'}
(\tilde{x}_i - x^K_i)
&= N_F^+( \lceil \lambda^K \rceil - \lambda^K); \\
\sum_{i: \ x^K_i + b_i = \lambda^K, \ i > i'}
(\tilde{x}_i - x^K_i)
&= (N_F(\lambda^K) - N_F^+) ( \lfloor \lambda^K \rfloor - \lambda^K)
\end{align*}
This implies that
\begin{align*}
\sum_{i \in N} (\tilde{x}_i - x^K_i)
&=
\sum_{i: \ x^K_i + b_i = \lambda^K} (\tilde{x}_i - x^K_i) \\
&=N_F^+( \lceil \lambda^K \rceil - \lambda^K)
+ (N_F(\lambda^K) - N_F^+) ( \lfloor \lambda^K \rfloor - \lambda^K) \\
&=N_F^+ \lceil \lambda^K \rceil 
+ (N_F(\lambda^K) - N_F^+) \lfloor \lambda^K \rfloor - N_F(\lambda^K) \lambda^K \\
&= N_F^+(\lceil \lambda^K \rceil - \lfloor \lambda^K \rfloor) + N_F(\lambda^K) (\lfloor \lambda^K \rfloor - \lambda^K) \\
&= N_F^+ - N_F^+ = 0.
\end{align*}
It follows that $\sum_{i \in N} \tilde{x}_i = \sum_{i \in N} x^K_i = R$ and thus that $\tilde{x}$ is feasible for $\tilde{Q}(K)$.

We show that $\tilde{x}$ is optimal for $\tilde{Q}(K)$ by demonstrating its optimality for $\tilde{Q}(K, \tilde{\lambda})$ for $\tilde{\lambda} := \frac{1}{2}(\lfloor \lambda^K \rfloor + \lceil \lambda^K \rceil)$. We do this by checking the optimality conditions in (\ref{eq_L_opt_Z}). Since $\tilde{\lambda} - \lfloor \tilde{\lambda} \rfloor = \frac{1}{2}$, we need to only consider the first, third, and fifth condition in (\ref{eq_L_opt_Z}):
\begin{description}
\item[Condition 1] If $\tilde{\lambda} \leq l_{i,j^K(i)} + b_i$, then also $\lambda^K \leq l_{i,j^K(i)} + b_i$ since $\lambda^K \leq \tilde{\lambda}$. It follows from (\ref{eq_x}) that $x^K_i = l_{i,j^K(i)}$ and thus $\tilde{x}_i = l_{i,j^K(i)}$.
\item[Condition 3] If $l_{i,j^K(i)} + b_i \leq \tilde{\lambda} \leq u_{i,j^K(i)} + b_i$, then also $l_{i,j^K(i)} + b_i \leq \lambda^K \leq u_{i,j^K(i)} + b_i$ since both $l_{i,j^K(i)} + b_i$ and $u_{i,j^K(i)} + b_i$ are integer. It follows from (\ref{eq_x}) that $x^K_i = \lambda^K - b_i$, which implies that either $\tilde{x}_i = \lfloor x^K_i \rfloor = \lfloor \lambda^K \rfloor - b_i = \lfloor \tilde{\lambda} \rfloor - b_i$ or $\tilde{x}_i = \lceil x^K_i \rceil = \lceil \lambda^K \rceil - b_i = \lceil \tilde{\lambda} \rceil - b_i$.
\item[Condition 5] If $\tilde{\lambda} \geq u_{i,j^K(i)} + b_i$, then also $\lambda^K \geq u_{i,j^K(i)} + b_i$ since $u_{i,j^K(i)} + b_i$ is integer. It follows from (\ref{eq_x}) that $x^K_i = u_{i,j^K(i)}$ and thus $\tilde{x}_i = u_{i,j^K(i)}$.
\end{description}
The candidate solution $\tilde{x}$ and multiplier $\tilde{\lambda}$ satisfy all conditions, meaning that $\tilde{x}$ is optimal for $\tilde{Q}(K)(\tilde{\lambda})$ and thus also for $\tilde{Q}(K)$. 

Recall that our end goal is to obtain the objective value of $\tilde{x}$ for $\sum_{i \in N} \phi(x_i + b_i)$. Note, that this value can now be expressed directly in terms of $\lambda^K$ and the bookkeeping parameters of $Q(K)$ as
\begin{align}
\tilde{V}^K &:= V_B(\lambda^K) + N_F^+ \phi(\lceil \lambda^K \rceil)
+ (N_F(\lambda^K) - N_F^+) \phi(\lfloor \lambda^K \rfloor) \nonumber \\
&=
\tilde{V}^K := V_B(\lambda^K) + N_F(\lambda^K)(\lambda^K - \lfloor \lambda^K \rfloor) \phi(\lceil \lambda^K \rceil)
+ (N_F(\lambda^K) - (N_F(\lambda^K)(\lambda^K - \lfloor \lambda^K \rfloor)) \phi(\lfloor \lambda^K \rfloor).
\label{eq_int_obj}
\end{align}
Given $V_B(\lambda^K)$, $N_F(\lambda^K)$, and $\lambda^K$ as output of the breakpoint search procedure in Algorithm~\ref{alg_final}, this computation takes $O(F)$ time and thus does not alter the worst-case time complexity of the algorithm.

We conclude this section with a note on the computational complexity of both RAP-DIBC and $\tilde{P}$ for quadratic objective functions. For fixed $m$, Algorithm~\ref{alg_final} outputs an optimal solution to RAP-DIBC with quadratic objective function in \emph{strongly} polynomial time when considering an algebraic tree computation model. Thereby, we add a new problem to the class of strongly polynomially solvable mixed-integer quadratic programming problems. If, additionally, we allow the floor operation in the considered computational model (see also the discussion in \cite{Hochbaum1994}), also the algorithm for $\tilde{P}$, i.e., with the adaptation in (\ref{eq_int_obj}) included, outputs an optimal solution to the integer problem $\tilde{P}$ in \emph{strongly} polynomial time for fixed $m$.

\section{Evaluation}
\label{sec_eval}

In this section, we asses the practical efficiency of Algorithm~\ref{alg_final}. We first focus on the performance of our approach on realistic instances of Min-Thres-EV. In a second step, we assess the scalability of our approach by evaluating them on synthetically generated instances of varying size. As far as we are aware, there are no tailored algorithms to solve RAP-DIBC, $\tilde{P}$, or one of their special cases with $m > 1$. Therefore, we compare the efficiency of our approach with that of the off-the-shelf solver Gurobi \cite{gurobi}. To increase the fairness of the comparison, we only consider quadratic objectives so that both approaches compute an optimal solution to the problem. Moreover, we consider only the continuous problem RAP-DIBC since initial testing suggested that there was no significant difference between the performance of our algorithm and that of Gurobi for RAP-DIBC and $\tilde{P}$. We implemented our algorithm in Python version 3.7 and integrated Gurobi using Gurobi's Python API; the corresponding code is accessible via \url{https://github.com/mhhschootuiterkamp/RAP_DIBC}. All simulations and computations are executed on a 2.80-GHz Dell Latitude 3420 with an Intel Core i7-6700HQ CPU and 16 GB of RAM.

Section~\ref{sec_instance} describes the construction of the instance sets that we will use to perform our computational experiments and in Section~\ref{sec_results}, we present and discuss our results.

\subsection{Instance generation and implementation details}
\label{sec_instance}

We generate two types of instances, namely instances of the minimum-threshold EV charging Min-Thres-EV and a set of randomly generated instances for the scalability evaluation. In both cases, we choose $\phi$ to be the quadratic function $\phi(y) = y^2$ to allow for a fairer comparison with the Gurobi implementation.

We first create a set of instances for Min-Thres-EV. For this, we consider a setting wherein an EV is empty and available for residential charging from 18:00 PM and must be fully charged by 8:00 AM on the next day. We divide this charging horizon of 14 hours into 15-minute time intervals, so that $T = 56$ and $\Delta t = \frac{1}{4}$. For the charging restriction, we use the Nissan Leaf as a reference EV \cite{Nissan2023}, meaning that $X^{\max} = 6.6$~kW and the maximum capacity of the EV battery is 39~kWh. Furthermore, confirming to the recommendations in \cite{ApostolakiIosifidou2017}, we set $X^{\min} = 1.1$~kW. To capture differences in power consumption profiles between households, we run our simulations using real power consumption measurement data of 40 households that were obtained in the field test described in \cite{Hoogsteen2017}. For each household, we simulate 300 charging sessions , where each session corresponds to a combination of a specific day (out of 100 days) and a specific charging requirement (out of three). The three different charging requirements that we consider correspond to charging 25\%, 50\%, or 100\% of the battery, respectively, meaning that we choose $R \in \lbrace 9,750; 19,500; 39,000 \rbrace$.

Finally, we create instances for the scalability comparison as follows. We first generate the vectors $\tilde{l}$ and $\tilde{u}$. For this, we draw $m-1$ random variables $X_2, \ldots, X_{m}$ and $m-2$ random variables $Y_2,\ldots,Y_{m-1}$ from the uniform distribution $U(0,1)$, initialize $\tilde{u}_1 := 2$, and set $\tilde{l}_j = \tilde{u}_{j-1} + X_j$ for $j \in \lbrace 2, \ldots, m \rbrace$ and $\tilde{u}_j = \tilde{l}_j + Y_j$ for $j \in \lbrace 2, \ldots, m-1 \rbrace$. Next, we generate the lower bounds $l_{1,1},\ldots, l_{n,1}$ and upper bounds $u_{1,m},\ldots, u_{n,m}$. We generate two sequences $W,Z$ of $n-1$ random variables from the uniform distribution $U(0;\frac{1}{n})$, initialize $l_{n,1} = 1$ and $u_{1,m} = l_{1,m} + 1 = \tilde{l}_m + 1$, and set $l_{i,1} = l_{i+1,1} - W_i$ and $u_{i+1,1} = u_{i,1} - Z_i$ for $i \in \lbrace 1,\ldots,n-1 \rbrace$. Note that by construction, this choice of parameters satisfies all four special cases (F1,L1), (F1,L2), (F2,L1), and (F2,L2) that we study in this paper and for which Algorithm~\ref{alg_final} is valid. We select the resource value $R$ from the uniform distribution $U(\sum_{i \in N} l_{i,1}, \sum_{i \in N} u_{i,m})$. Note that this choice of $R$ ensures feasibility of the problem by Lemma~\ref{lemma_feasible}. Finally, we set $b_n = 0$ and $b_i = b_{i+1} + V_i$ for $i \in \lbrace 1, \ldots, n-1 \rbrace$, where each $V_i$ is a random variable drawn from $U(0,1)$.

Since we are not aware of other tailored algorithms for RAP-DIBC, we compare the performance of our algorithm to that of the off-the-shelf solver Gurobi version 10.0.1 \cite{gurobi}. Here, we implement the disjoint interval bound constraints (\ref{eq_box}) as follows. For each combination of variable index $i \in N$ and interval $j \in M$, we introduce a binary variable $y_{i,j}$ that is one if $x_i$ lies in the $j^{\text{th}}$ interval, i.e., $l_{i,j} \leq x_i \leq u_{i,j}$, and zero otherwise. For each $i \in N$, we add the constraint $\sum_{j \in M} y_{i,j} = 1$ to ensure that at most one interval is selected. Lastly, for each $i \in N$, we add a constraint that sets the correct lower and upper bounds on $x_i$ depending on which of its corresponding binary variables equals one, i.e., $\sum_{j \in M} l_{i,j} y_{i,j} \leq x_i \leq \sum_{j \in M} u_{i,j} y_{i,j}$ for each $i \in N$.

We generate instances for different values of $n$ and $m$. Initial testing confirmed the time complexity analysis in Section~\ref{sec_multiple} that the running time of our algorithm increases drastically with the value of $m$. Therefore, we decide to run simulations for $m \in \lbrace 2,3,4 \rbrace$ with the following values for $n$:
\begin{itemize}
\item For $m=2$: $n \in \lbrace 10; 20; 50; 100; 200; 500; \ldots ;  10,000; 20,000; 50,000; 100,000 \rbrace$;
\item For $m=3$:  $n \in \lbrace 10; 20; 50; 100; 200; 500; 1,000; 2,000; 5,000; 10,000 \rbrace$;
\item For $m=4$:  $n \in \lbrace 10; 20; 50; 100; 200; 500; 1,000 \rbrace$.
\end{itemize}
For each of these combinations of $m$ and $n$, we generate and solve ten instances. For the Gurobi implementation, we set the maximum solving time to one hour.

\subsection{Results}
\label{sec_results}
In this section, we present and discuss the results of the evaluation. We first focus on the performance of Algorithm~\ref{alg_final} on instances of Min-Thres-EV. Table~\ref{tab_EV} shows the mean execution times of Algorithm~\ref{alg_final} and Gurobi on these instances, split out by charging requirement. Moreover, Figure~\ref{plot_EV_box} shows for each charging requirement the boxplot of the ratios between the execution times of Gurobi and our algorithm. Figure~\ref{plot_EV_box} indicates that, on averag our algorithm is fourteen to fifteen times as fst as Gurobi. The results in Table~\ref{tab_EV} suggest that the execution times of both our algorithm and the Gurobi implementation decrease slightly as the charging requirement increases. Despite this, the execution times and the relative difference in execution time between Algorithm~\ref{alg_final} and Gurobi are in the same order of magnitude for each charging requirement. Finally, we note that Algorithm~\ref{alg_final} solves the realistic instances of Min-Thres-EV in the order of milliseconds. Common speed and delay requirements for communication networks in DEM systems are significantly higher than this \cite{Deshpande2011}. This means that our algorithm is suitable for integration in such systems since it is unlikely that it will be the main (computational) bottleneck.

\begin{table}[ht!]
\centering
\begin{tabular}{r | r r}
\toprule
$R$ & Algorithm~\ref{alg_final} & Gurobi \\
\midrule
$0.25 \cdot 39,000$ & $4.53 \cdot 10^{-4}$ & $2.09 \cdot 10^{-2}$ \\
$0.5 \cdot 39,000$ & $3.51 \cdot 10^{-4}$ & $1.48 \cdot 10^{-2}$ \\
$39,000$ & $3.20 \cdot 10^{-4}$ & $1.49 \cdot 10^{-2}$ \\
\bottomrule
\end{tabular}
\caption{Mean execution times (s) for Algorithm~\ref{alg_final} and Gurobi for each charging requirement.}
\label{tab_EV}
\end{table}

\begin{figure}[ht!]
\centering
\includegraphics{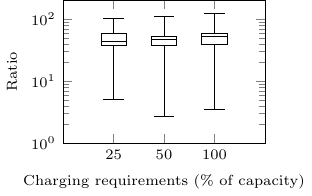}
\caption{Boxplots of the ratios of the execution times between Gurobi and Algorithm~\ref{alg_final} for each charging requirement.}
\label{plot_EV_box}
\end{figure}

Figure~\ref{plot_scale} shows the results of the scalability evaluation. To further visualize the dependency of the execution time on $n$, we fit a power law, i.e., a function $f(n) = c_1 \cdot n^{c_2}$ to the execution times of Algorithm~\ref{alg_final}. Furthermore, Table~\ref{tab_m} shows the mean execution times for each considered value of $m$ and $n$. For the case $m=2$ and $n=1,000$, all execution times of Gurobi exceeded the time limit of 3,600 seconds and thus no mean value is presented. The power laws in Figure~\ref{plot_scale} show that the execution time of Algorithm~\ref{alg_final} grows linearly in $n$ for $m=2$, quadratically for $m=3$, and cubicly for $m=4$. This suggests that in practice the algorithm is a factor $O(\log n)$ faster than the theoretical worst-case time complexity suggests. On the other hand, the execution times of Gurobi do not change much for different values of $m$.

\begin{figure}[ht!]
\begin{subfigure}[t]{\textwidth}
\includegraphics{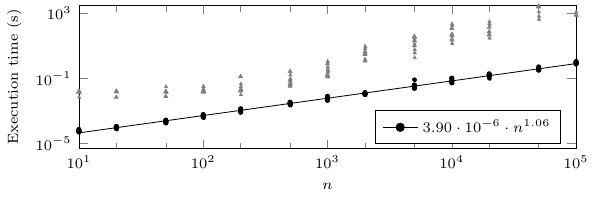}
\caption{$m=2$.}
\label{plot_scale_m2}
\end{subfigure}
\begin{subfigure}[t]{.55\textwidth}
\includegraphics{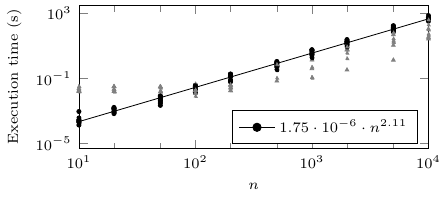}
\caption{$m=3$.}
\label{plot_scale_m3}
\end{subfigure}
\begin{subfigure}[t]{.44\textwidth}
\includegraphics{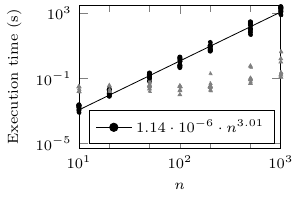}
\caption{$m=4$.}
\label{plot_scale_m4}
\end{subfigure}
\caption{Execution times of Algorithm~\ref{alg_final} (circles, black) and Gurobi (tirangles, gray).}
\label{plot_scale}
\end{figure}

\begin{table}[ht!]
\centering
\begin{subtable}[t]{0.32\textwidth}
\centering
\resizebox{\columnwidth}{!}{%
\begin{tabular}[t]{r | r r }
\toprule
$n$ & Algorithm~\ref{alg_final} & Gurobi \\
\midrule
10 & $5.83 \cdot 10^{-5}$ & $1.45 \cdot 10^{-2}$ \\
20 & $9.81 \cdot 10^{-5}$ & $1.02 \cdot 10^{-2}$ \\
50 & $2.19 \cdot 10^{-4}$ & $1.31 \cdot 10^{-2}$ \\
100 & $4.81 \cdot 10^{-4}$ & $1.98 \cdot 10^{-2}$ \\
200 & $1.09 \cdot 10^{-3}$ & $3.36 \cdot 10^{-2}$ \\
500 & $2.88 \cdot 10^{-3}$ & $9.00 \cdot 10^{-2}$ \\
1,000 & $6.02 \cdot 10^{-3}$ & $4.00 \cdot 10^{-1}$ \\
2,000 & $1.14 \cdot 10^{-2}$ & $3.93$ \\
5,000 & $3.74 \cdot 10^{-2}$ & $1.86 \cdot 10^{01}$ \\
10,000 & $7.55 \cdot 10^{-2}$ & $7.86 \cdot 10^{01}$ \\
20,000 & $1.56 \cdot 10^{-1}$ & $1.05 \cdot 10^{03}$ \\
50,000 & $3.83 \cdot 10^{-1}$ & $2.47 \cdot 10^{03}$ \\
100,000 & $9.20 \cdot 10^{-1}$ & - \\
\bottomrule
\end{tabular}}
\caption{$m=2$.}
\label{tab_m2}
\end{subtable}
\hspace{.2pt}
\begin{subtable}[t]{0.32\textwidth}
\centering
\resizebox{\columnwidth}{!}{%
\begin{tabular}[t]{r | r r }
\toprule
$n$ & Algorithm~\ref{alg_final} & Gurobi \\
\midrule 
10 & $3.07 \cdot 10^{-4}$ & $2.04 \cdot 10^{-2}$ \\
20 & $1.13 \cdot 10^{-3}$ & $2.15 \cdot 10^{-2}$ \\
50 & $5.45 \cdot 10^{-3}$ & $2.26 \cdot 10^{-2}$ \\
100 & $2.87 \cdot 10^{-2}$ & $2.68 \cdot 10^{-2}$ \\
200 & $1.32 \cdot 10^{-1}$ & $2.82 \cdot 10^{-2}$ \\
500 & $7.47 \cdot 10^{-1}$ & $5.00 \cdot 10^{-1}$ \\
1,000 & $4.15$ & $1.28$ \\
2,000 & $1.64 \cdot 10^{01}$ & $6.56$ \\
5,000 & $1.27 \cdot 10^{02}$ & $4.18 \cdot 10^{01}$ \\
\hphantom{0}10,000 & $4.65 \cdot 10^{02}$ & $1.36 \cdot 10^{02}$ \\
\bottomrule
\end{tabular}}
\caption{$m=3$.}
\label{tab_m3}
\end{subtable}
\hspace{.2pt}
\begin{subtable}[t]{0.32\textwidth}
\centering
\resizebox{\columnwidth}{!}{%
\begin{tabular}[t]{r | r r }
\toprule
$n$ & Algorithm~\ref{alg_final} & Gurobi \\
\midrule 
10 & $1.78 \cdot 10^{-3}$ & $2.31 \cdot 10^{-2}$  \\
20 & $1.12 \cdot 10^{-2}$ & $2.04 \cdot 10^{-2}$ \\
50 & $1.50 \cdot 10^{-1}$ & $4.30 \cdot 10^{-2}$ \\
100 & $1.14$ & $2.80 \cdot 10^{-2}$ \\
200 & $8.35$ & $4.65 \cdot 10^{-2}$ \\
500 & $1.66 \cdot 10^{02}$ & $2.20 \cdot 10^{-1}$ \\
\hphantom{00}1,000 & $1.78 \cdot 10^{03}$ & $1.25$ \\
\bottomrule
\end{tabular}}
\caption{$m=4$.}
\label{tab_m4}
\end{subtable}
\caption{Mean execution times (s) for Algorithm~\ref{alg_final} and Gurobi.}
\label{tab_m}
\end{table}

The results in Table~\ref{tab_m2} show that for the case $m=2$ our algorithm outperformed the Gurobi implementation by two orders of magnitude in almost all considered cases. For $m=3$, Table~\ref{tab_m3} indicates that Gurobi is on average faster from $n=500$ onward and for $m=4$, the results in Table~\ref{tab_m4} imply that only in the cases $n=10$ and $n=20$ our algorithm is on average faster than Gurobi. This suggests that our algorithm is to be preferred for either $m=2$ or for small values of $m$ and $n$ when considering quadratic objectives. However, on the other hand, it also suggests that the complexity improvement in Algorithm~\ref{alg_final} as compared to the initial Algorithm~\ref{alg_init} also leads to a speed-up in practice. Finally, we expect that our approach becomes more competitive when considering non-quadratic objective functions, especially when these cannot be easily linearized or approximated. This is because such problems cannot be solved by specialized mixed-integer convex quadratic solvers anymore. In particular, as demonstrated in Section~\ref{sec_alg_init}, we may not simply solve the quadratic version of the problem and employ a reduction result like Lemma~\ref{lemma_reduction} to conclude that this solution is also optimal for the non-quadratic objective function. In contrast the only potential increase in execution time in our approach is the evaluation of this new objective function in Lines~7, 13, 15,18, 21, 31, 37, 41, and 47 of Algorithm~\ref{alg_final}.

\section{Conclusions}
\label{sec_concl}

In this paper, we consider a resource allocation problem with a symmetric separable convex objective function and structured disjoint interval bound constraints, motivated by electric vehicle charging problems in decentralized energy management (DEM). We present an algorithm that solves four special cases of this problem in  $O \left(\binom{n+m-2}{m-2} (n \log n + nF) \right)$ time, where $m$ is the number of disjoint intervals for each variable and $nF$ represents the number of flops required for one evaluation of the objective function. Our algorithm solves the continuous and integral versions of the problem simultaneously without an increase in computational complexity. Computational experiments indicate that the algorithm is fast enough in practice for successful application in DEM systems. Although generally an increase in the number of intervals $m$ leads to a large increase in execution time, our algorithm still outperforms a general-purpose solver by one or two orders of magnitude for small $m$.

We conclude this paper with several directions for future research. First, our eventual solution approach and algorithm have the same worst-case time complexity for each of the four considered special cases of bound constraints. This may suggest that all these four cases are equally difficult. However, we have also shown that the complexities of finding feasible solutions to these cases differ, with one case being significantly more difficult than the other three. It would be interesting to see whether this fact may be used to derive more efficient algorithms or tighten the complexity bound of Algorithm~\ref{alg_final} for these three other cases.

Finally, one direction for future research is to identify more special cases of the problem that can be solved efficiently. Furthermore, it would be interesting to see whether additional allocation constraints could be incorporated, such as nested or even general submodular constraints. In particular, additional nested constraints would model the scheduling of battery charging with a minimum (dis)charging threshold. Since batteries are expected to play a large role in the ongoing energy transition, such an extension would contribute greatly to the integration of these devices in low-voltage grids.

\bibliographystyle{plain}
\bibliography{lib_onoff}
\end{document}